\documentclass[reqno,11pt,a4paper]{amsart}
\usepackage{graphicx} % Required for inserting images

\usepackage[toc]{appendix}
\usepackage[T1]{fontenc}
\usepackage{enumerate} 
\usepackage{amsmath,amsfonts,amssymb,mathrsfs }
\usepackage[latin1]{inputenc}
\usepackage[english]{babel}
\usepackage{mathtools}
\usepackage{braket}
\usepackage{color}
\usepackage{hyperref}
\usepackage{xfrac, nicefrac}
\usepackage{csquotes}
\usepackage{esint}
\usepackage{todonotes}
\usepackage{orcidlink}

\newtheorem{theorem}{Theorem}[section]
\newtheorem{lemma}[theorem]{Lemma}

\newtheorem{definition}[theorem]{Definition}
\newtheorem{remark}[theorem]{Remark}
\newtheorem*{remark*}{Remark}
\newtheorem*{definition*}{Definition}
\newtheorem*{claim*}{Claim}
\newtheorem*{property*}{Property}

\numberwithin{equation}{section}

\usepackage{hyperref}

\newcommand{\R}{\mathbb{R}}
\newcommand{\N}{\mathbb{N}}

\newcommand{\mint}{\mathop{\int\hskip -1,05em -\, \!\!\!}\nolimits}

\def\A{\mathcal A}

\def\d{\delta}

\def\e{\varepsilon}

\title{Regularity results for elliptic equations on cones}
\author[Antonini]{Carlo Alberto Antonini \orcidlink{0000-0002-7663-1090}}  \address{Carlo Alberto Antonini \\ 
Dipartimento di Matematica ``Federico Enriques'', 
Universit\`a degli studi di Milano, 
Viale Cesare Saldini 50, 20133,
Milan,
Italy\\ ORCID ID: 0000-0002-7663-1090}
\email{\url{carlo.antonini@unimi.it}\\ \url{antonini@altamatematica.it}}

\author{Filomena Pacella}
\address{F. Pacella. Dipartimento di Matematica ``Guido Castelnuovo'', Sapienza Universit\`a di Roma,  P.le Aldo Moro 2, 00185 Roma, Italy}
\email{\url{filomena.pacella@uniroma1.it}}

\author{Camilla Chiara Polvara}
\address{C.C. Polvara. Dipartimento di Matematica ``Guido Castelnuovo'', Sapienza Universit\`a di Roma,  P.le Aldo Moro 2, 00185 Roma, Italy}
\email{\url{camilla.polvara@uniroma1.it}}

\author{Luigi Provenzano}
\address{L. Provenzano.  Dipartimento di Scienze di Base e Applicate per l'Ingegneria ``SBAI'', Sapienza Universit\`a di Roma, Via Antonio Scarpa, 14, 00161 Roma, Italy}
\email{\url{luigi.provenzano@uniroma1.it}}

\date{}

\begin{document}

\begin{abstract}
We study global regularity of solutions to Dirichlet or Neumann elliptic problems in spherical sectors $S_{D,R}$ of radius $R>0$ in $\R^N, N\ge 2$, where $D$ is the bounded domain on the unit sphere $\mathbb{S}^{N-1}$ which spans the spherical sector. One of the main results shows that boundedness of the gradient of the solutions of Poisson equations holds whenever $\lambda_1(D)\ge N-1$, where $\lambda_1(D)$ is the first nontrivial eigenvalue of the Laplace Beltrami operator $-\Delta_{\mathbb{S}^{N-1}}$ on the domain $D$ with Dirichlet or Neumann boundary conditions on $\partial D$.

As an example of Maz'ya shows, the condition on the eigenvalue is sharp.

For general spherical sectors and for $p$-Laplacian equations, $p>1$  we prove weighted global lipschitzianity of the solutions, as well as second order regularity.
 \end{abstract}


\maketitle

\section{Introduction}\label{sec:intro}

Let $D\subset \mathbb{S}^{N-1}, N\geq  2$ be a  smooth domain on the unit sphere and consider the cone
\begin{equation}\label{def:cone}
    \Sigma_D:=\big\{x=r\,\theta:\,r>0,\,\theta\in D\big\},
\end{equation}
that we assume to have a Lipschitz continuous boundary.

In this paper we investigate first and second order regularity for solutions of linear and quasilinear elliptic equations on a spherical sector $\Sigma_D\cap B_R,$ where $B_R$ is a ball in $\R^N$, of radius $R>0$ and centered at the origin, for simplicity we take $R=2$.

We begin our analysis with the Poisson equation, considering  both the Neumann boundary value problem
\begin{equation}\label{eq:lapl:neu}
    \begin{cases}
        -\Delta u=f\quad &\text{in $\Sigma_D\cap B_2$}
        \\
        \partial_\nu u=0&\text{on $\partial\Sigma_D\cap B_2$,}
    \end{cases}
\end{equation}
where $\nu$ is the outward unit normal to $\partial \Sigma_D$, and the Dirichlet boundary value problem
\begin{equation}\label{eq:lapl:dir}
    \begin{cases}
        -\Delta u=f\quad &\text{in $\Sigma_D\cap B_2$}
        \\
        u=0&\text{on $\partial\Sigma_D\cap B_2$.}
    \end{cases}
\end{equation}
Due to the presence of the singularity at the vertex of $\Sigma_D$, a first relevant question is to understand whether or not the gradient of $u$ is bounded, at least for a right-hand side $f$ good enough.

For Poisson equations in convex domains $\Omega \subset \mathbb{R}^N$,
the boundedness of the derivatives of the solutions has been obtained
in several papers for the Dirichlet problem 
\cite{Lady53,LU68,M09,M10,M11,M12} and by Maz'ya for the Neumann problem in
\cite{M09}. In our case, this means that if $D$ is a convex domain on
the unit sphere, then $|\nabla u|\in L^\infty(\Sigma_D \cap B_1)$ if $f \in L^q(\Sigma_D \cap B_2)$, $q>N$, for
any solution $u$ of \eqref{eq:lapl:neu}
or \eqref{eq:lapl:dir}.
Thus, the question is to study the boundedness of the gradient when
$D$ is not convex. An interesting counterexample of Maz'ya, cited in
(\cite{M09}, Section 3), shows that there are non-convex domains
$D \subset \mathbb{S}^{N-1}$ for which the Neumann problem,
\eqref{eq:lapl:neu}, admits solutions with unbounded gradient in
$\Sigma_D \cap B_1$, even if
$f \in L^\infty(\Sigma_D \cap B_1)$.

However, looking closely at
Maz'ya example, one can observe that a crucial property of $D$,
which is responsible for the unboundedness of the gradient, is that
the first nontrivial Neumann eigenvalue $\lambda_1(D)$ of the
Laplace--Beltrami operator
$-\Delta_{\mathbb{S}^{N-1}}$ on the domain $D$ is smaller than
$N-1$.

Note that if $D$ is a convex domain on the sphere, then
$\lambda_1(D) \geq N-1$ (\cite{AM, Escobar, M09}).
Thus, the real question is whether the boundedness of the gradient of
the solution can be obtained by assuming
$\lambda_1(D) \geq N-1$, even when $D$ is not convex. This is one of
the main contributions of our paper that we obtain, among other
results, in the theorem that we state below.

Let $\lambda_i$, $Y_i$ denote the Neumann (resp. Dirichlet) ordered eigenvalues and eigenfunctions of
$-\Delta_{\mathbb{S}^{N-1}}$ on $D$, for $i=0,1,\ldots,k,\ldots$ ( $i=1,2,\ldots,k,\ldots$ in the Dirichlet
case); see Section \ref{sec:laplneu} for definitions and notation.

\begin{theorem}\label{thm}
    Let $D\subsetneq\mathbb{S}^{N-1}$ be a domain of class $C^2$, and such that $\Sigma_D$ has Lipschitz continuous boundary. Let  $u\in W^{1,2}(\Sigma_D\cap B_2)$ be a weak solution to the Neumann problem \eqref{eq:lapl:neu}  (resp. the Dirichlet problem \eqref{eq:lapl:dir}).

    Let $q\ge\frac{2N}{N+2}$ ($q>1$ if $N=2$) be such that, upon setting $\eta\equiv \eta_q:=2-\frac{N}{q}$, the quantity $\eta\,(\eta+N-2)$ is not a Neumann (resp. Dirichlet) eigenvalue of the Laplace-Beltrami operator $-\Delta_{\mathbb S^{N-1}}$ on $D$.

  Then, if $f\in L^q(\Sigma_D\cap B_2)$, we have that $u$ admits the finite expansion
    \begin{equation}   \label{decomposition0}     
    u=u_0+\sum_{\substack{i\geq 0\\\xi_i<\eta}}c_i r^{\xi_i}Y_i(\theta),\quad \bigg(\text{resp. }u=u_0+\sum_{\substack{i\geq 1\\\xi_i<\eta}}c_i r^{\xi_i}Y_i(\theta)\bigg)\quad\text{in $\Sigma_D\cap B_1$,}
        \end{equation}

    where $c_i\in\R$, $u_0\in W^{2,q}(\Sigma_D\cap B_1)$ and the exponents $\xi_i$ are given by:
$$
\xi_i:=\frac{-(N-2)+\sqrt{4\lambda_i+(N-2)^2}}{2},
$$ 
and the sum in \eqref{decomposition0} may be empty.
Additionally, we have that
\begin{enumerate}
    \item  if $\lambda_1>N\big(2-\frac{N}{q}\big)\big(1-\frac{1}{q}\big)$, then $u\in W^{2,q}(\Sigma_D\cap B_1)$,
    \item  if $q>N$ then $u_0(0)=0$ and $\nabla u_0(0)=0$,
    \item if $q>N$ and $\lambda_1= N-1$, then $u\in C^{0,1}(\overline\Sigma_D\cap B_1)$,

\item if $q>N$ and $\lambda_1> N-1$, then $u\in C^{1,\alpha}(\overline\Sigma_D\cap B_1)$ and $\nabla u(0)=0$, where $\alpha=\alpha(\lambda_1,q,N)\in (0,1)$.

\end{enumerate}

\end{theorem}

\begin{remark}
    \rm{Given $\eta=2-\frac{N}{q}$, then the condition $\eta(\eta+N-2)$ is not a Neumann (resp. Dirichlet) eigenvalue of the Laplacian $-\Delta_{\mathbb S^{N-1}}$ is  satisfied for every $q>N$ whenever the first eigenvalue satisfies $\lambda_1(D)\geq 2N$.  Otherwise, the condition can fail only for finitely many values of $q$. 
    Indeed, if $q>N$, then $\eta(\eta+N-2)<2N$. Hence, if $\lambda_1(D)\geq 2N$, the quantity $\eta(\eta+N-2)$ cannot coincide with any eigenvalue. If instead $\lambda_1(D)< 2N$, the discreteness of the spectrum implies that there are only finitely many eigenvalues strictly below $2N$. Since each eigenvalue can correspond to at most one value of $q$, the condition can fail for at most finitely many values of $q$. }
\end{remark}

The proof of Theorem \ref{thm} makes use of the spectral decomposition obtained by Dauge \cite{dauge_cones}, especially for what concerns the regular part $u_0$ in \eqref{decomposition0}. In Section \ref{sec:laplneu} we show that the singular part of the solutions can be written  explicitly in terms  of eigenvalues and eigenfunctions of the related boundary value problem on the spherical domain $D\subsetneq \mathbb{S}^{N-1}$.

To the best of our knowledge, this explicit relation between the regularity of solutions  on cones and the spectrum of $-\Delta_{\mathbb{S}^{N-1}}$ on  $D$, and in particular the identification of $\lambda_1=N-1$ as the sharp threshold for Lipschitz regularity, has not appeared before.

For related results on spectral decompositions and the asymptotic behavior of solutions to linear boundary-value problems near conical singularities, we refer to the works of Kondrat'ev \cite{Kondratev} and Maz'ya-Plamenevski\u{\i} \cite{MP1,MP2,MP3}, as well as to the monograph by Kozlov-Maz'ya \& Rossmann \cite{KozlovMazya}.

% \begin{remark}\rm{
%     The case $N=2$ is special, since in this case $\partial \Sigma_D\cap B_2$ is made of two segments, and in particular its second fundamental form is zero (away from the vertex). Hence, in this case, the gradient boundedness and second order regularity of solutions is retrieved with the same proof as in the convex case \cite{M09, cia}.}
% \end{remark}
\vspace{0.1cm}

To see why the size of $\lambda_1(D)$ is what really matters in the regularity of solutions near conical singularities, consider the simple case of a $H^1$-harmonic function $u$ in $\Sigma_D\cap B_1$ satisfying Dirichlet or Neumann conditions on the lateral boundary $\partial\Sigma_D\cap B_1$. We briefly describe the Dirichlet case (the discussion does not change for Neumann). We have that $g:=u|_{\Sigma_D\cap \partial B_1}$ is in $H^{1/2}(\Sigma_D\cap \partial B_1)$, and we can write $g=\sum_{i=1}^{\infty}c_i v_i$, where $c_j\in\mathbb R$ and $v_i$ are the (traces of) the eigenfunctions of the Steklov-Dirichlet problem
\begin{equation}\label{SD}
\begin{cases}
\Delta v=0\,, & {\rm in\ }\Sigma_D\cap B_1\,,\\
\partial_\nu v=\xi v\,,  & {\rm on\ }\Sigma_D\cap\partial B_1\,,\\
v=0\,, & {\rm on\ }\partial\Sigma_D\cap B_1.
\end{cases}
\end{equation}
 If $\{\lambda_i(D)\}_{i=1}^{\infty}$ denote the eigenvalues of the Dirichlet Laplacian on $D$ and $\{Y_i(\theta)\}_{i=1}^{\infty}$ a corresponding $L^2(D)$-orthonormal basis of eigenfunctions, then the eigenvalues of \eqref{SD} are given by $\{\xi_i\}_{i=1}^\infty$ with
 $$
\xi_i=\frac{-(N-2)+\sqrt{4\lambda_i(D)+(N-2)^2}}{2}
 $$
 and a corresponding (orthogonal) basis of harmonic functions in $H^1(\Sigma_D\cap B_1)$ is given by $\{v_i\}_{i=1}^{\infty}=\left\{r^{\xi_i}Y_i(\theta)\right\}_{i=1}^{\infty}$. Hence $g=\sum_{i=1}^{\infty}c_iv_i$  simply reads $\sum_{i=1}^{\infty}c_iY_i$. It is a well-known fact (see e.g., \cite{auchmuty_steklov}) that the harmonic function $u$ has the series representation:
\begin{equation}\label{Steklov expansion}
u=\sum_{i=1}^{\infty}c_ir^{\xi_i}Y_i(\theta)=\sum_{i=1}^{\infty}c_iv_i
 \end{equation}
 and that $\sum_{i=1}^{\infty}c_i^2(1+\xi_i)<\infty$ (i.e., the sum converges in $H^1(\Sigma_D\cap B_1)$). The most important observation here is that, from \eqref{Steklov expansion}, one immediately sees that the regularity at the origin ($C^{0,\alpha}$, $C^{k,\alpha}$, $W^{k,p}$, etc.) is driven by the exponent $\xi_1$ (provided the coefficient $c_1\neq 0$), and therefore by the first eigenvalue $\lambda_1(D)$. In particular, we have Lipschitz continuity at the origin if $\xi_1\geq 1$ (i.e., when $\lambda_1(D)\geq N-1$), while the differentiability of solutions is guaranteed  if $\lambda_1(D)>N-1$ (in some very specific cases also when $\lambda_1(D)=N-1$, see Remark \ref{finalrmk}).  In the Neumann case, in the series expansion we have also the constant eigenfunction $v_0={\rm const}$ with eigenvalue $\lambda_0(D)=0$, thus the regularity is again driven by $\lambda_1(D)$ which is the first positive eigenvalue.

The analysis of the regularity of harmonic functions in $H^1(\Sigma_D\cap B_1)$ near the singularity is simpler than the general case. However, in the general case of $-\Delta u=f$, we will see that the regularity of $u$ is recast to that of the harmonic functions in the cone. More precisely, in Section \ref{sec:laplneu}, building on the general results of \cite{dauge_cones}, we will see  that a compactly supported solution of $-\Delta u=f$ with $f\in L^q(\Sigma_D)$ on $\Sigma_D$ (with Dirichlet or Neumann conditions) can be written as the sum of a regular part, which has the same regularity of the smooth case, and a finite number of harmonic functions that behave exactly as the eigenfunctions $v_i$. In other words, the harmonic part of the solution encodes the geometry of the cone through the eigenvalues of $D$, and in turn determines the regularity of the solution.
\vspace{0.2cm}

We now move onto nonlinear operators; more precisely, given $p>1$, we study solutions $u\in W^{1,p}(\Sigma_D\cap B_2)$ to either the Neumann boundary value problem 
\begin{equation}\label{eq:bdryNeu}
        \begin{cases}
        -\Delta_p u=f\quad&\text{in $\Sigma_D\cap B_2$}
        \\
        \partial_\nu u=0 \quad&\text{on $\partial\Sigma_D\cap B_2$}
    \end{cases}
\end{equation}
or the Dirichlet problem
\begin{equation}\label{eq:bdryDir}
    \begin{cases}
        -\Delta_p u=f\quad&\text{in $\Sigma_D\cap B_2$}
        \\
        u=0 \quad&\text{on $\partial\Sigma_D\cap B_2$}
    \end{cases}
\end{equation}
where $\Delta_pu\equiv \mathrm{div}\big(|\nabla u|^{p-2}\nabla u \big)$ is the $p$-Laplace operator.

We shall often use the shorthand notation
\begin{equation}\label{stress:field}
    \A(\xi):=|\xi|^{p-2}\,\xi,\quad\xi\in \R^N,
\end{equation}
so that  $\Delta_p u=\mathrm{div}\big(\A(\nabla u) \big)$. The vector field $\A(\nabla u)$ is typically called \textit{Stress field}. 

We recall that $u$ is a weak solution to \eqref{eq:bdryNeu} (resp \eqref{eq:bdryDir}) if
\begin{equation*}
    \int_{\Sigma_D\cap B_2} \A(\nabla u)\cdot \nabla \varphi\,dx=\int_{\Sigma_D\cap B_2} f\,\varphi\,dx
\end{equation*}
for every $\varphi\in W^{1,p}(\Sigma_D\cap B_2)$ compactly supported in $B_2$ (resp. $\varphi\in W^{1,p}_0( \Sigma_D\cap B_2)$ compactly supported in $B_2$).

A few comments are in order. As already mentioned, the proof of Theorem
\ref{thm} relies heavily on the spectral properties of the Laplace operator
\cite{dauge_cones,KozlovMazya} and therefore fundamentally exploits its
linearity.

%\todo[inline]{CA: ho aggiunto l'articolo di Tolksdorf. Lui tratta Dirichlet per il $p$-Laplaciano+ soluzioni nonnegative. Controllate anche voi questa frase, perch\'e va in parte in contrasto con quanto detto sopra "To the best of our knowledge, this explicit relation between the regularity of solutions  on cones and the spectrum of $-\Delta_{\mathbb{S}^{N-1}}$ on  $D$ (in particular the identification of $\lambda_1=N-1$ as the sharp threshold for Lipschitz regularity) has not appeared before."}

%\todo[inline]{In generale, la parte sotto \'e quasi tutta nuova. Date un'occhiata se vi piace}
For $p$-Laplace type equations, Tolksdorf \cite{T83} obtained related regularity results for nonnegative solutions to Dirichlet problems \eqref{eq:bdryDir} in domains with conical boundary points, under suitable pointwise assumptions on the right-hand side, by means of maximum principles and Hopf-type lemmas.

More recent results, based on suitable integral inequalities
\cite{AC251,ACP25,BDMS22,CMS26,CM11,CM14,cia,DfP23,MMS25,SSV25,SV26}
have shown that global Lipschitz and second-order regularity for solutions to the
nonlinear problems \eqref{eq:bdryNeu}-\eqref{eq:bdryDir} still hold when the
reference domain $\Omega$ is convex.

When convexity is absent, sufficient conditions can instead be formulated in
terms of suitable integrability assumptions on the curvatures of
$\partial\Omega$; see \cite{AC251,CM11,CM14}. Such assumptions, however, are
not generally satisfied by Lipschitz cones, even when the cone is generated by
a smooth spherical sector $D$--see Remark \ref{Remark:derseccono} below.

% In order to obtain global Lipschitz regularity of solution, the convexity assumption  can be replaced by a suitable integrability assumption on the curvatures of $\partial \Omega$; namely $\partial \Omega\in W^{2}L^{N-1,1}$ see \cite{AC251, CM11, CM14}.
% However, any cone $\Sigma_D$ lacks such regularity- see Remark \ref{Remark:derseccono} below.

Nevertheless, by adopting a different approach that exploits the scaling invariance of the cone $\Sigma_D$, we are able to establish weighted global Lipschitz and second-order regularity for solutions in this setting.

% Nevertheless, by using a different approach which takes advantage of the scaling property of the cone $\Sigma_D$, we are able to prove in any case \textit{weighted} global Lipschitz and second order regularity of solutions. 

To this end, given $\Omega\subset \R^N$ open, and a nonnegative measurable function $\rho$ on $\Omega$, we will denote by $L^q(\Omega; \rho dx)$ the set of measurable functions $v:\Omega\to \R^\ell$, $\ell\in \N$, for which the norm
\[
\|v\|_{L^q(\Omega; \rho dx)}:=\bigg(\int_\Omega |v|^q\,\rho dx \bigg)^{1/q}
\]
is finite. When $\rho\equiv 1$, we shall simply write $L^q(\Omega)$.

Our result concerning gradient regularity of solutions to \eqref{eq:bdryNeu}-\eqref{eq:bdryDir} is the following \begin{theorem}\label{thm:nonlingrad}
    Let $D\subsetneq \mathbb{S}^{N-1}$ be a domain of class $C^{1,\beta}$, $\beta\in (0,1)$, and such that $\Sigma_D$ is a Lipschitz cone. Let $u\in W^{1,p}(\Sigma_D\cap B_2)$ be  solution to either the Neumann problem \eqref{eq:bdryNeu} or the Dirichlet problem \eqref{eq:bdryDir}, and assume that $f$ satisfies the integrability condition
\begin{equation}\label{f:integr1}
    f\in L^{\max\{N/p,\,1\}+\delta}(\Sigma_D\cap B_2) \cap L^{N+\d}\big(\Sigma_D\cap B_2;\,|x|^{(p-1)N+\delta p}\,dx\big)
\end{equation}
%     and
% \begin{equation}\label{f:integr2}
%     \|f\|_{L^{N+\d}\big(\Sigma_D\cap B_2;\,|x|^{(p-1)N+\delta p}\,dx\big)}:=\Big(\int_{\Sigma_D\cap B_2} |x|^{(p-1)N+\delta p}\,|f|^{N+\d}\,dx\Big)^{1/(N+\d)}<\infty,
% \end{equation}
 for some small $\d>0$. Then there exists $\alpha\in (0,1)$ such that 
\begin{equation}\label{weight:xnablau}
    \text{the map $x\mapsto |x|\,\nabla u(x)$ belongs to $L^{\infty}(\Sigma_D\cap B_1)$, }
\end{equation}
    and such that
    \begin{equation}\label{weigh:holdnablau}
    \text{ the map $x\mapsto |x|^{1+\alpha}\,\nabla u(x)$ is of class $C^{0,\alpha}(\Sigma_D\cap B_1)$. }
\end{equation}
\end{theorem}

%Moreover, there exists a constants $C_0=C_0(N,p,\d,\mathrm{diam}(D), \mathcal L_D,\|\partial D\|_{C^{1,1}})>0$  such that
%\begin{equation}\label{stima:quant1}
%\begin{split}
  %  \Big\||x|&\,\nabla u(x)\Big\|_{L^\infty(\Sigma\cap B_1)}+\Big[|x|^{1+\alpha}\nabla u(x)\Big]_{C^{0,\alpha}(\Sigma\cap B_1)}
  %  \\
  %  &\leq C_0\,\bigg\{ \|u\|_{W^{1,p}(\Sigma\cap B_2)}+\Big[\|f\|_{L^{N/p+\delta}(\Sigma\cap B_2)}+\|f\|_{L^{N+\d}\big(\Sigma\cap B_2;\,|x|^{(p-1)N+\delta p}\,dx\big)}\Big]^{\frac{1}{(p-1)}}\bigg\}.
  %  \end{split}
%\end{equation}

For what concerns second order regularity, this is better expressed in terms of the Stress field $\A(\nabla u)$ given by \eqref{stress:field}. 

\begin{theorem}\label{thm:weightH1}
    Let $D\subsetneq \mathbb{S}^{N-1}$ be a domain of class $C^{2,\beta}$, $\beta\in (0,1)$, and such that $\Sigma_D$ is a Lipschitz cone. Assume that $u\in W^{1,p}(\Sigma_D\cap B_2)$ is a weak solution to either \eqref{eq:bdryNeu} or \eqref{eq:bdryDir}, with 
\begin{equation}\label{ass:f2}
    f\in L^{\max\{N/p,1\}+\delta}(\Sigma_D\cap B_2)\cap  L^2\big( \Sigma_D\cap B_2;|x|^{2p}\,dx\big)
\end{equation}
    for some small $\delta>0$. Then $\A(\nabla u)\in W^{1,2}_{\rm{loc}}\big((\Sigma_D\cap B_2)\setminus \{0\}\big)$ and 
    \begin{equation}\label{thesis:Sob}
        \nabla\A(\nabla u)\in L^2\big(\Sigma_D \cap B_1;\,|x|^{2p}dx\big).
    \end{equation}
    \end{theorem}
%     there exists a positive constant $C_0$ depending on $N,p,\d,\mathrm{diam}(D),\mathcal{L}_{D},\|\partial D\|_{C^{1,1}}$ such that
% \begin{equation}
% \begin{split}
%     \big\|&\nabla \A(\nabla u)\big\|_{L^2(\Sigma \cap B_1;\,|x|^{2p}dx)}
%     \\
%     &\leq C\,\bigg(N,p,\delta,\mathrm{diam}(D), \mathcal{L}_D,\|D\|_{C^{1,1}},\|u\|_{W^{1,p}(\Sigma\cap B_2)}, \|f\|_{L^{\max\{N/p,1\}+\delta}(\Sigma\cap B_2)},\|f\|_{L^{2}(\Sigma\cap B_2;\, |x|^{2p}dx)} \bigg).
%     \end{split}
% \end{equation}

\begin{remark}\label{Remark:derseccono}
\rm{
Let $\Omega\subset\mathbb{R}^N$ be a bounded domain, and denote by
$\mathcal{B}_\Omega(x)$ the second fundamental form of $\partial\Omega$ at
$x\in\partial\Omega$. A sufficient condition for solutions to
\eqref{eq:bdryNeu}-\eqref{eq:bdryDir} to be globally Lipschitz continuous is
$\mathcal{B}_\Omega\in L^{N-1,1}(\partial\Omega)$; see
\cite{AC251,CM11,CM14} for the corresponding results and for the definition
of the relevant Lorentz spaces $L^{N-1,q}$.

As for global $W^{2,2}$-regularity, a weaker sufficient condition is
$\mathcal{B}_\Omega\in L^{N-1,\infty}(\partial\Omega)$, together with the
smallness assumption
\begin{equation}\label{condition}
    \lim_{r\to0^+}\sup_{x_0\in\partial\Omega}
    \|\mathcal{B}_\Omega\|_{L^{N-1,\infty}
    (\partial\Omega\cap B_r(x_0))}
    <\kappa_0,
\end{equation}
where $\kappa_0>0$ is a sufficiently small constant depending on
$N$, $\operatorname{diam}\Omega$, and the Lipschitz character
$\mathcal{L}_\Omega$ of $\Omega$; see \cite{ACCFM25,cia}.

Suppose now that, in a neighborhood of the origin, $\Omega$ agrees with the
Lipschitz cone $\Sigma_D$ generated by a smooth domain
$D\subsetneq\mathbb{S}^{N-1}$. Denoting by $\mathcal{B}_D$ the second
fundamental form of $\partial D$ in $\mathbb{S}^{N-1}$, a simple computation
based on the scaling properties of the cone and its curvatures yields
\begin{equation}\label{stima:BS}
    |\mathcal{B}_{\Sigma_D}(x)|
    \leq
    C(N,\mathcal{L}_{\Sigma_D})
    \frac{\|\mathcal{B}_D\|_{L^\infty(\partial D)}}{|x|}.
\end{equation}
Consequently,
$\mathcal{B}_{\Sigma_D}\in
L^{N-1,\infty}_{\mathrm{loc}}(\partial\Sigma_D)$, whereas, unless the cone is
flat,
$\mathcal{B}_{\Sigma_D}\notin
L^{N-1,1}_{\mathrm{loc}}(\partial\Sigma_D)$. Thus, the above sufficient
condition for global Lipschitz regularity fails at the vertex.

On the other hand, estimate \eqref{stima:BS}, together with the scale
invariance of the weak-$L^{N-1}$ norm, implies that
$$
    \|\mathcal{B}_{\Sigma_D}\|_{L^{N-1,\infty}
    (\partial\Sigma_D\cap B_r)}
    \leq
    C(N,\mathcal{L}_{\Sigma_D})
    \|\mathcal{B}_D\|_{L^\infty(\partial D)}
$$
for every $r>0$. Therefore, the smallness condition \eqref{condition}, and
hence the corresponding global $W^{2,2}$-regularity result, applies provided
that $\|\mathcal{B}_D\|_{L^\infty(\partial D)}$ is sufficiently small. In view
of property $(1)$ in Theorem \ref{thm}, it would be interesting to
investigate whether this curvature smallness condition can be related to the
first eigenvalue $\lambda_1(D)$.
}
\end{remark}

\begin{remark}
\rm{
Following the proof of Theorem \ref{thm:nonlingrad} in Section \ref{sec:nonlin},
we believe that, if one is interested only in the weighted Lipschitz estimate
\eqref{weight:xnablau}, the assumption
\[
f\in L^{N+\delta}\bigl(\Sigma_D\cap B_2;\,
|x|^{(p-1)N+\delta p}\,dx\bigr)
\]
can likely be relaxed to a suitable condition formulated in terms of a weighted
Lorentz space of type $L^{N,1}$, in the spirit of the sharp regularity results
of \cite{CM11,CM14,KM12,KM13,KM14,KM141,KM18}.

Moreover, the assumption $\partial D\in C^{2,\beta}$ in
Theorem \ref{thm:weightH1} is purely technical. In view of
Remark \ref{Remark:derseccono}, it should be possible to replace it with the
weaker assumption $\partial D\in W^2L^{N-1,\infty}$, together with a suitable
smallness condition as in \eqref{condition}.
Furthermore, Theorems \ref{thm:nonlingrad}-\ref{thm:weightH1} could be extended to a broader class of nonlinear operators, along the lines of \cite{A26}.

For the sake of conciseness, we do not pursue these refinements here.
}
\end{remark}

Finally, coming back to the  linear problem, and the condition $\lambda_1(D) \geq N-1$, which represents
a sharp threshold for the Lipschitz continuity of any solution, it is
interesting to point out that a similar condition appears in other,
apparently unrelated, problems in cones. One of these is the critical
exponent problem in the cone $\Sigma_D$:
\begin{equation}\label{critical equation}
\begin{cases}
-\Delta u = u^{2^*-1} & \text{in } \Sigma_D,\\
\dfrac{\partial u}{\partial \nu}=0 & \text{on } \partial\Sigma_D,
\end{cases}
\qquad
2^*=\frac{2N}{N-2}, \quad N\geq 3,
\end{equation}
for which it has been proved in \cite{LPT} that if $D$ is convex, then the standard ``bubble''
\[
U(x)=\frac{C}{(1+|x|^2)^{\frac{N-2}{2}}},
\]
(where $C$ is a suitable constant) is, up to rescaling, the only
positive solution of \eqref{critical equation} (see \cite{CFR} for more general critical equations).
It has been proved in \cite{CPP} that this
characterization does not hold whenever
$\lambda_1(D)<N-1$, and actually nonradial positive solutions exist. 

Note that  Theorem \ref{thm} combined with the Kelvin invariance of the problem \eqref{critical equation},  implies any weak solution $u$ of \eqref{critical equation} have globally bounded derivatives in the whole $\Sigma_D$; see, e.g., \cite{PPP25}.

Another, more geometrical, question is the characterization of the
minimizers for the relative isoperimetric inequality in cones, which are the
spherical sectors $\Sigma_D\cap B_R$, $R>0$, if $D$ is convex 
\cite{CRS,FI,LP,RR},
while they are other nonradial domains when
$\lambda_1(D)<N-1$  \cite{IPW}.

Similar symmetry results hold in convex cones also for the
characterization of constant mean curvature hypersurfaces in cones and for domains in $\Sigma_D$ which admit solutions for the
overdetermined torsion problem \cite{PT}, while break of
symmetry occurs whenever
$\lambda_1(D)<N-1$  \cite{IPW}.

In view of these results, it can be conjectured that for the above
problems, the same symmetry characterizations which hold for convex
cones should be extended to the case $\lambda_1(D)\geq N-1$, but
so far no proofs have been obtained and the question seems difficult
to attack.

Therefore, it is particularly remarkable that the regularity result
of Theorem 1.1 shows that $\lambda_1(D)\geq N-1$ is the sharp
condition which allows one to extend to solutions in nonconvex
spherical sectors the same regularity holding in the convex
framework. In other words, the convexity assumption can be removed as
long as $\lambda_1(D)$ stays above (or is equal) to $N-1$.

This is the first result of this kind, to our knowledge. 

The paper is
organized as follows.
    In Section \ref{sec:auxiliary}, we introduce the notation and collect some preliminary results. In Section \ref{sec:nonlin}, we prove Theorems \ref{thm:nonlingrad} and \ref{thm:weightH1}, while Section \ref{sec:laplneu} is devoted to the proof of Theorem \ref{thm}.

\section{Notation and preliminary  results}\label{sec:auxiliary}
Given $p>1$, we set 
\begin{equation*}
    p^*:=\begin{cases}
        \frac{Np}{N-p}\quad &\text{if $1<p<N$}
       \\
        \text{any number $>2$} \quad&\text{if $p\geq N$.}
 \end{cases}
\end{equation*}
We also write $p'=p/(p-1)$ to denote the H\"older conjugate of $p$.

For a measurable set $\Omega$ with positive measure, we will denote by $\mathrm{diam}(\Omega)$ its diameter, by $|\Omega|$ its Lebesgue measure, and
\begin{equation*}
    \mint_\Omega v\,dx:=\frac{1}{|\Omega|} \int_\Omega v\,dx
\end{equation*}
whenever $v$ is a measurable function on $\Omega$, either nonnegative or integrable.

Given $\Omega\subset \R^N$ a bounded open set, we write $W^{1,p}(\Omega)$ for the Sobolev space of functions $v$ whose distributional gradient $\nabla v\in L^p(\Omega)$. The set   $W^{1,p}_c(\Omega)$ will denote the space of functions $v\in W^{1,p}(\Omega)$ which are compactly supported in $\Omega$, while $W^{1,p}_0(\Omega)$ is the closure of $C^{\infty}_c(\Omega)$ in the $W^{1,p}$-norm $\|v\|_{W^{1,p}}=\|v\|_{L^p(\Omega)}+\|\nabla v\|_{L^p(\Omega)}$. When $p=2$, we will simply write $W^{1,2}(\Omega)=H^1(\Omega), W^{1,2}_0(\Omega)=H^1_0(\Omega),W^{1,2}_c(\Omega)=H^1_c(\Omega) $. We write $\mathrm{supp}\, \phi$ to denote the support of a function $\phi$.

The space $C^{0,\alpha}(\Omega)$, $\alpha\in (0,1]$, will denote the space of $\alpha$-H\"older continuous functions $v:\Omega\to \R^d$, that is they satisfy $\|v\|_{C^{0,\alpha}(\Omega)}\coloneqq \|v\|_{L^\infty(\Omega)}+ [v]_{C^{0,\alpha}(\Omega)}<\infty$, where we set
\begin{equation*}
    [v]_{C^{0,\alpha}(\Omega)}\coloneqq \sup_{\substack{x,y\in\Omega\\ x\neq y}}\frac{|v(x)-v(y)|}{|x-y|^\alpha}
\end{equation*}

 We will also use the following  shorthand notation: given a cone $\Sigma\subset \R^N$ as in \eqref{def:cone}, we will write
\begin{equation}\label{short:hand}
    \Sigma_{r,R}:=\Sigma\cap \big(B_{R}\setminus B_r\big),\quad \partial\Sigma_{r,R}:=\partial\Sigma\cap \big(B_{R}\setminus B_r\big)
\end{equation}
for any couple of radii $0<r<R$.

% We will also need the following lemma.
% \begin{lemma}\label{lemma:holder:annuli}
% Let $\Sigma\subset \R^N$ be a cone with vertex at the origin, and suppose that a given vector field $H:\Sigma\to \R^d$, $d\in \N$ fulfills the H\"older continuity condition
% \begin{equation}\label{ass:z1}
%     [H]_{C^{0,\alpha}(\Sigma_{R/2,R})}\leq c_0\quad\text{for all $R\leq 1$,}
% \end{equation}
% for some $\alpha \in (0,1]$ and constant $c_0>0$. Then $V\in C^{0,\alpha}(\overline{\Sigma}\cap B_1)$, with estimate
% \begin{equation*}
%     [H]_{C^{0,\alpha}(\Sigma_{R/2,R})}\leq C_2(N,c_0).
% \end{equation*}

% \end{lemma}

Let us now recall the definition of the sets we shall be dealing with. 

\begin{definition}[Lipschitz and $C^{1,\alpha}$-domains]\label{def:lip}
\rm{An open, connected set $\Omega$ in $\R^N$ is called a Lipschitz domain if 
 there exist  constants $L_\Omega>0$ and $R_\Omega \in (0, 1)$ 
such that, for every $x_0\in \partial \Omega$ and $R\in (0, R_\Omega]$ there exist an isometry $T$ such that $T(x_0)=0$,  and an $L_\Omega$-Lipschitz continuous function 
$\phi : B'_{R}\to (-\ell, \ell)$, where 
\begin{equation}\label{ell}
\ell = R (1+L_\Omega),
\end{equation}
satisfying $\phi(0')=0$ , and
\begin{equation}\label{may100}
\begin{split}
    &T(\partial \Omega) \cap \big(B'_{R}\times (-\ell,\ell)\big)=\{(x', \phi (x'))\,:\,x'\in B'_{R}\},
    \\
    & T(\Omega) \cap \big(B'_{R}\times (-\ell,\ell)\big)=\{(x',x_n)\,:\,x'\in B'_{R}\,,\,\phi (x')<x_n<l\},
\end{split}
\end{equation}
where $B'_R$ is the $(N-1)$-dimensional ball of radius $R>0$ and centered at $0'\in\R^{N-1}$.
Moreover, we set
\begin{equation}\label{may101}
\mathcal L_\Omega = (L_\Omega, R_\Omega),
\end{equation}
and call $\mathcal L_\Omega$ a Lipschitz characteristic of $\Omega$.

If in addition $\phi$ is of class $C^{1,\alpha}(B'_R)$ for some $\alpha\in (0,1]$, then we say that $\partial \Omega$ is of class $C^{1,\alpha}$, and we define the norm $\|\partial\Omega\|_{C^{1,\alpha}}$as follows:
given $\{x_i\}_{i=1}^N \subset \partial\Omega$ such that $B_{R_\Omega}(x_i)$ form an open cover of
$\partial\Omega $, 
and let $\{\phi_i\}_{i=1}^N=\{\phi_{x_i}\}_{i=1}^N$ be coordinate charts satisfying \eqref{may100}, we set
\begin{equation}\label{normadeom}
\|\partial\Omega\|_{C^{1,\alpha}}
:= \sup_{i=1,\dots,N} \|\phi_i\|_{C^{1,\alpha}}.
\end{equation}
}
\end{definition}

It is well known  that $\Omega$ is of class $C^{1,1}$ if and only if it admits a weak second fundamental form $\mathcal{B}\in L^\infty(\partial \Omega)$, with estimate
\begin{equation}\label{curvature:twosided}
    \|\mathcal{B}\|_{L^\infty(\partial \Omega)}\leq C_1\,\|\partial\Omega\|_{C^{1,1}},
\end{equation}
where the constant $C_1>0$ depends on $N,\mathrm{diam}(\Omega), \mathcal{L}_\Omega$ (see for instance \cite{A24}, in particular Eq. (2.10) there).

 By simply passing to local coordinates, analogous definition and observations can be extended to domains $\Omega$ on the sphere $\mathbb{S}^{N-1}$. The following observation will be frequently used.

\begin{remark}\label{remark:lipconst}
\rm{ Let $D$ be a domain on the sphere $\mathbb{S}^{N-1}$ such that $\Sigma=\Sigma_D$ is a Lipschitz cone around the origin. Then, by passing to polar coordinates, it is easy to see that, for all $0<r<R$, the Lipschitz characteristics  $\mathcal{L}_{\Sigma_{r,R}}=(L_{\Sigma_{r,R}}, R_{\Sigma_{r,R}})$ of $\Sigma_{r,R}$ depend only on $N,r,R$ and the Lipschitz characteristics $\mathcal{L}_D=(L_D,R_D)$ of $D$. Moreover, we can estimate  the diameter $\mathrm{diam}(\Sigma_{r,R})$ via  $N,r,R$ and $\mathrm{diam}(D)$, and the Lebesgue norm
\begin{equation}\label{mis:sigma}
    c(N,s,t,\mathcal{L}_D)\,R^N\leq|\Sigma\cap (B_{tR}\setminus B_{sR})|\leq C(N,s,t,\mathcal{L}_D)\,R^N,\quad\text{$0\leq s<t$.}
\end{equation}

If in addition $D$ is of class $C^{1,\beta}$, $\beta\in (0,1]$, then recalling \eqref{short:hand}, we have that $\partial \Sigma_{r,R}$ is of class $C^{1,\beta}$, with $\|\partial \Sigma_{r,R}\|_{C^{1,\beta}}$ depending on $N,r,R,\mathcal{L}_D$  and $\|\partial D\|_{C^{1,\beta}}$.}
\end{remark}

\subsection{Sobolev inequalities on annuli}
Let us recall  the following Sobolev inequality \cite[Eq. (5.4)]{cia} (see also \cite[Lemma 2.1]{GM23}, or \cite[Lemma 3.6]{ACP25}).

\begin{lemma}\label{lem:sobannuli}
   Let $R\in (0,1]$, and let $v\in W^{1,2}(B_{2R}\setminus B_R)$. Then there exists a positive constant $C=C(N)$ such that, for every $\epsilon>0$ and for every  $\sigma,\tau$ satisfying
   \[
   R\leq \sigma<\tau\leq 2R,
   \]
   the following Sobolev-Poincar\'e  inequality holds
   \begin{equation}
       \int_{B_\tau\setminus B_\sigma} v^2\,dx\leq C\,\epsilon^2\,R^2\,\int_{B_\tau\setminus B_\sigma} |\nabla v|^2\,dx+\frac{C}{\epsilon^N\,(\tau-\sigma)\,R^{N-1}}\,\bigg( \int_{B_\tau\setminus B_\sigma} |v|\,dx\bigg)^2.
   \end{equation}
\end{lemma}

We shall also need a version of Lemma \ref{lem:sobannuli} localized around the boundary of a Lipschitz domain $\Omega$. To this end, we need to introduce some notation. For any $x_0\in \partial\Omega$, by Definition \ref{def:lip} we may find an isometry $T$, with $T(x_0)=0$, and an $L_\Omega$ Lipschitz map $\phi$ satisfying \eqref{may100}. Thus,  for all $x\in B_{R_\Omega}(x_0)$, it is well defined the boundary flattening map $\mathcal{F}$ 
\begin{equation}\label{flatten}
    \mathcal{F}(x',x_n)=\big(y',y_n-\phi(y')\big),\quad y=Tx,
\end{equation}
whose inverse is given by $\mathcal{F}^{-1}(y',y_n)=(x',x_n+\phi(x'))$, $x=T^{-1}y$. In particular, by Definition \ref{def:lip}, we have $\mathcal{F}(\Omega\cap B_{R_\Omega}(x_0))\subset \{x_n>0\}$ and  $\mathcal{F}(\partial\Omega\cap B_{R_\Omega}(x_0))\subset \{x_n=0\}$; moreover, by the $L_\Omega$-Lipschitz continuity of $\phi$, we have
\begin{equation}\label{bound:flattn}
    \|\nabla \mathcal{F}\|_\infty+\|\nabla \mathcal{F}^{-1}\|_\infty\leq C(N)(1+L_\Omega),
\end{equation}
 and since $|\phi(x')|\leq L_\Omega\,|x'|$ as $\phi(0')=0$, it is immediate to see  that $\mathcal{B}_{R_\Omega/\underline{C}(1+L_\Omega)}\subset B_{R_\Omega}$ for some constant $\underline{C}=\underline{C}(N)>0$.

Accordingly, we define the new ``boundary adapted balls'' 
\begin{equation}\label{flatten:balls}
    \mathcal{B}_r\equiv \mathcal{B}_r(x_0):=\mathcal{F}^{-1}\big(B_r\big).
\end{equation}

\begin{figure}[ht]\label{fig1}
\centering
\includegraphics[width=1\textwidth]{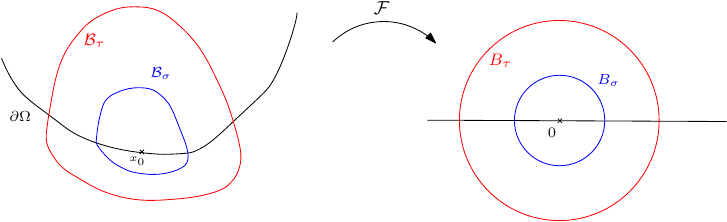}
\caption{The boundary adapted balls $\mathcal{B}_r$.}
\end{figure}

Under this notation, we now establish a boundary Sobolev inequality on annuli, which will follow by combining Lemma \ref{lem:sobannuli} with a flattening-reflection argument.

\begin{lemma}[Sobolev inequality on boundary annuli]\label{lemm:bdr:sobannuli}
    Let $\Omega\subset \R^N$ be a Lipschitz domain, with Lipschitz characteristics $\mathcal{L}_\Omega=(L_\Omega,R_\Omega)$. Let $x_0\in \partial\Omega$, $R\leq R_\Omega/\big(2\underline{C}(1+L_\Omega)\big)$, and let  $v\in W^{1,2}\big(\Omega\cap (\mathcal{B}_{2R}\setminus \mathcal{B}_{R})\big)$. Then there exists a constant $\underline{C}=\underline{C}(N,L_\Omega)>0$ such that
    \begin{equation}\label{eq:sob:bdryannuli}
               \int_{\Omega\cap (\mathcal{B}_\tau\setminus \mathcal{B}_\sigma)} v^2\,dx\leq C\,\epsilon^2\,R^2\,\int_{\Omega\cap(\mathcal{B}_\tau\setminus \mathcal{B}_\sigma)} |\nabla v|^2\,dx+\frac{C}{\epsilon^N\,(\tau-\sigma)\,R^{N-1}}\,\bigg( \int_{\Omega\cap(\mathcal{B}_\tau\setminus \mathcal{B}_\sigma)} |v|\,dx\bigg)^2    \end{equation}
    holds true for every $\epsilon>0$, and for every $R\leq \sigma<\tau\leq 2R$.
\end{lemma}

\begin{proof}
Define  $w=v\circ\mathcal{F}^{-1}$, which belongs to $W^{1,2}\big((B_{2R}\setminus B_R)\cap \{x_n>0\}\big)$, and consider its even extension
\[
\overline{w}(y',y_n):=\begin{cases}
    w(y',y_n)\quad &y_n>0
    \\
    w(y',-y_n)\quad &y_n<0,
\end{cases}
\]
which belongs to $W^{1,2}(B_{2R}\setminus B_R)$. By applying Lemma \ref{lem:sobannuli} to $\overline{w}$, and using that for any $r\geq 1$
\[
\|\overline{w}\|_{L^r(B_\tau\setminus B_\sigma)}=C(r)\,\|w\|_{L^r\big(\{x_n>0\}\cap (B_\tau\setminus B_\sigma)\big)},\qquad \int_{B_\tau\setminus B_\sigma}|\nabla\overline{w}|^2dx=2\,\int_{(B_\tau\setminus B_\sigma)\cap \{x_n>0\}}|\nabla w|^2dx,
\]
we deduce
\begin{equation*}
\begin{split}
    \int_{(B_\tau\setminus B_\sigma)\cap \{x_n>0\}}|w|^2\,dx\leq &\,C(N)\,\e^2\,R^2\,\int_{(B_\tau\setminus B_\sigma)\cap \{x_n>0\}}|\nabla w|^2\,dx
    \\
    &+\frac{C(N)}{\e^N\,(\tau-\sigma)\,R^{N-1}}\,\bigg( \int_{(B_\tau\setminus B_\sigma)\cap \{x_n>0\}}  |w|\,dx\bigg)^2,
    \end{split}
\end{equation*}
from which \eqref{eq:sob:bdryannuli} follows via a change of variables $y=\mathcal{F}(x)$, and taking also into account \eqref{bound:flattn}.
\end{proof}

\subsection{The regularized Stress field}
In the course of the proof, we shall use a regularized version of \eqref{stress:field}: for $\e>0 $, let 
\begin{equation}\label{regul:stress}
    \A_\e(\xi):=\big[\e^2+|\xi|^2 \big]^{\frac{p-2}{2}}\,\xi,\quad \xi\in \R^N.
\end{equation}
We also set
\begin{equation}\label{def:aebeBe}
    a_\e(t):=\big(\e^2 +t^2\big)^{\frac{p-2}{2}},\quad b_\e(t):=a_\e(t)\,t,\quad B_\e(t):=\int_0^t b_\e(s)\,ds.
\end{equation}

A simple computation also shows that
\begin{equation}\label{unif:ineps}
\min\{p-1,1\}\leq i_{b_\e}:=\inf_{t>0}\frac{b_\e'(t)t}{b_\e(t)}\leq \sup_{t>0}\frac{b_\e'(t)t}{b_\e(t)}=:s_{b_\e}\leq\max\{p-1,1\}
\end{equation}
for every $\e\in (0,1)$.

Moreover,  it is well known that (see, e.g., \cite{CFV16})
\begin{equation*}
    c(N,p)\,a_\e\big(|\xi|\big)\,\mathrm{Id}\leq \nabla_\xi \A_\e(\xi)\leq C(N,p)\,a_\e\big(|\xi|\big)\,\mathrm{Id},\quad \text{in the sense of matrices}
\end{equation*}
for all $\xi\in \R^N$, where $\mathrm{Id}$ is the identity matrix on $\R^N$.

Then, for every $C^2$-function $w$, by the chain rule and  \cite[Eq. (3.53)]{ACCFM25} (with $b(t)=b_\e(t)$ and matrix $M=\nabla^2w$), we find
\begin{equation}\label{elm:trace}
    \begin{split}\mathrm{tr}\Big(\big[\nabla \A_\e(\nabla w)\big]^2 \Big)&=\mathrm{tr}\Big(\big[\nabla_\xi \A_\e(\nabla w)\,\nabla^2 w\big]^2 \Big)
    \\
    &\geq c(N,p)\,\Big|\nabla_\xi \A_\e(\nabla w)\,\nabla^2 w \Big|^2=c(N,p)\,\big|\nabla \A_\e(\nabla w ) \big|^2.
    \end{split}
\end{equation}

\section{Proof of Theorems \ref{thm:nonlingrad}-\ref{thm:weightH1}}\label{sec:nonlin}

In order to prove Theorem \ref{thm:nonlingrad}, we start with the following lemma, which encodes global boundedness for $u$ and a Morrey type estimate around the origin for the function $|x|\,\nabla u$.

\begin{lemma}\label{lemma:energia}
  Let $D\subsetneq \mathbb{S}^{N-1}$ be a domain such that $\Sigma=\Sigma_D$ is a Lipschitz cone around the origin.  Let $u_0\in W^{1,p}(\Sigma\cap B_2)$ be a weak solution to either  \eqref{eq:bdryNeu} or \eqref{eq:bdryDir}, and assume that $f$ satisfies the integrability assumption \eqref{f:integr1}. Assume also that
    \begin{equation}\label{en:bddd}
        \|u_0\|_{W^{1,p}(\Sigma\cap B_2)}+\|f\|_{L^{   \max\{N/p,1\}+\delta}(\Sigma\cap B_2)}\leq 2.
    \end{equation}

    Then there exists $C_0=C_0(N,p,\delta,\mathcal{L}_D)>0$ such that
\begin{equation}\label{tesi:linf}
    \|u_0\|_{L^\infty\big(\Sigma\cap B_{7/4 }\big)}\leq C_0,
\end{equation}
    and the following estimate \begin{equation}\label{tesi:energia}
\mint_{\Sigma\cap B_{3/2R}} |x|^p|\nabla u_0|^p\,dx\leq C_0
    \end{equation}
   holds true for all $R\in (0,1]$.
\end{lemma}

\begin{proof}
    We start by proving \eqref{tesi:linf}: since $\Sigma$ is a Lipschitz cone, we may appeal to the De Giorgi-Nash-Moser regularity theory for boundary value problems \cite[Sections 2.6-2.7]{LU68}, \cite[Section 6.5]{G03} (see also \cite[Section 3, Remark 7.4]{A26}) which, by also taking into account Remark \ref{remark:lipconst}, gives
    \begin{equation*}
        \|u_0\|_{L^\infty(\Sigma\cap B_{7/4})}\leq C\big(N,p,\mathcal{L}_D,\delta, \|f\|_{L^{\max\{N/p,1\}+\delta}(\Sigma\cap B_2)}, \|u_0\|_{W^{1,p}(\Sigma\cap B_2)} \big),
    \end{equation*}
which together with \eqref{en:bddd} proves \eqref{tesi:linf}.  

Next we take a cut-off function $\eta\in C^\infty_c(B_{2R})$  such that
\begin{equation}\label{pr:eta}
  \mathrm{supp}(\eta)\Subset B_{7/4 R}, \quad 0\leq \eta\leq 1,\quad    \eta\equiv 1\quad\text{on $B_{3/2R}$,}\quad \|\nabla \eta\|_\infty\leq \frac{C(N)}{R}.
\end{equation}
Then,  we estimate
\begin{equation}\label{temp:0}
    \begin{split}
     \mint_{\Sigma\cap B_{7/4 R}}(\eta\,|x|)^p&\,|\nabla u_0|^p\,dx=\mint_{\Sigma\cap B_{7/4 R}}(\eta\,|x|)^p|\nabla u_0|^{p-2}\nabla u_0\cdot \nabla u_0\,dx
          \\&=\mint_{\Sigma\cap B_{7/4 R}}|\nabla u_0|^{p-2}\nabla u_0\cdot \nabla \big(u_0 (\eta\,|x|)^p \big)\,dx
          \\
          &\qquad\qquad-\mint_{\Sigma\cap B_{7/4 R}}|\nabla u_0|^{p-2}\nabla u_0\cdot \nabla \big( (\eta\,|x|)^p \big)\,u_0\,dx=:(I)+(II).
    \end{split}
\end{equation}
 Testing \eqref{eq:bdryNeu}- \eqref{eq:bdryDir} with function $u_0\,(\eta\,|x|)^p$, and  by exploiting \eqref{tesi:linf} and \eqref{pr:eta},  we obtain
\begin{equation*}
    \begin{split}|(I)|&\leq \mint_{\Sigma\cap B_{7/4 R}}|f|\,(\eta\,|x|)^p\,|u_0|\,dx\leq \|u_0\|_{L^\infty(\Sigma\cap B_{7/4})}\,\mint_{\Sigma\cap B_{2R}}|f|\,|x|^p\,dx
    \\
    &\leq C(N,p,\delta,\mathcal{L}_D)\,R^{p-N}\int_{\Sigma\cap B_{2R}}|f|\,dx,
    \end{split}
\end{equation*}
where in the last inequality we used \eqref{mis:sigma}.
Thus, in the case $p\geq N$, by \eqref{en:bddd} and since $R\leq 1$, we get $|(I)|\leq C(N,p,\delta,\mathcal{L}_D) $, while in the case $p<N$ we 
use H\"older's inequality with exponents $\Big(\tfrac{N}{p}+\delta,\tfrac{N+\delta\,p}{N+p(\delta-1)}\Big)$, \eqref{en:bddd} and that $R\leq 1$, and find
\begin{equation}\label{temp:1}
     \begin{split}|(I)|
    &\leq C'(N,p,\delta,\mathcal{L}_D)\, R^{p-N+N\big(\frac{N+p(\delta-1)}{N+p    \delta}\big)}\,\|f\|_{L^{N/p+\delta}(\Sigma\cap B_{2R})}
    \\
    &= C'(N,p,\delta,\mathcal{L}_D)\,R^{\frac{\delta\,p^2}{N+p\delta}}\,\|f\|_{L^{N/p+\delta}(\Sigma\cap B_{2R})}\leq C''(N,p,\delta,\mathcal{L}_D) .
    \end{split}
\end{equation}

Then we have
\begin{equation}\label{temp:2}
    \begin{split}
     |(II)|&\stackrel{\eqref{pr:eta}}{\leq} C(N,p,\mathcal{L}_D)\,\|u_0\|_{L^\infty(\Sigma\cap B_{7/4 R})}\,\bigg\{\mint_{\Sigma\cap B_{7/4 R}} |\nabla u_0|^{p-1}\big(|x|^{p}\,R^{-1}\,\eta^{p-1}+|x|^{p-1}\,\eta^p \big)\,dx    \bigg\} 
     \\
     &\stackrel{\eqref{tesi:linf}}{\leq} C'(N,p,\delta, \mathcal{L}_D)\,\mint_{\Sigma\cap B_{7/4 R}} |x|^{p-1}\,|\nabla u_0|^{p-1}\,\eta^{p-1}\,dx
     \\
     &\leq \e\, \mint_{\Sigma\cap B_{7/4 R}} (\eta\,|x|)^p\,|\nabla u_0|^{p}\,dx+C(N,p,\delta, \mathcal{L}_D,\e)
    \end{split}
\end{equation}
for all $\e\in (0,1)$, where in the second inequality we  used that $|x|^{-1}\geq (7/4 R)^{-1}$, and in the third one we exploited weighted Young's inequality. Thereby combining \eqref{temp:0}-\eqref{temp:2}, and  by choosing $\e=\e(N,p,\d,\mathcal{L}_D)$ small enough to reabsorb terms, we obtain
\begin{equation*}
    \mint_{\Sigma\cap B_{7/4 R}} (\eta\,|x|)^p\,|\nabla u_0|^{p}\,dx\leq C(N,p,\d,\mathcal{L}_D),
\end{equation*}
which by the properties of $\eta$ in \eqref{pr:eta} yields \eqref{tesi:energia}, that is our thesis.
\end{proof}

\begin{proof}[Proof of Theorem \ref{thm:nonlingrad}]
   Let $u$ be a weak solution to either \eqref{eq:bdryNeu} or \eqref{eq:bdryDir}. We first normalize the norms of $u$ and $f$, respectively; namely, we define the new functions 
    \begin{equation}\label{def:u0}
        u_0:=\frac{u}{\big(\|u\|_{W^{1,p}(\Sigma\cap B_2)}+\|f\|^{1/(p-1)}\big)}\quad\text{and}\quad f_0:=\frac{f}{\big(\|u\|_{W^{1,p}(\Sigma\cap B_2)}+\|f\|^{1/(p-1)}\big)^{p-1}},
    \end{equation}
  where we denoted by $\|f\|\equiv \|f\|_{L^{\max\{N/p,1\}+\delta}(\Sigma\cap B_2)}+\|f\|_{L^{N+\d}\big(\Sigma\cap B_2;\,|x|^{(p-1)N+\delta p}\,dx\big)}$ for brevity. Evidently, we have 
\begin{equation}\label{bounds}
    \|u_0\|_{W^{1,p}(\Sigma\cap B_2)}+\|f_0\|\leq 2,
\end{equation}
and by the homogeneity of the $p$-Laplace operator, we have that $u_0\in W^{1,p}(\Sigma\cap B_2)$ is a weak solution to
\begin{equation}\label{eq:u0}
    -\Delta_pu_0=f_0\quad\text{in $\Sigma\cap B_2$,}
\end{equation}
with the same homogeneous boundary condition as $u$ on $\partial \Sigma\cap B_2$. 
Next we apply a rescaling procedure,  that is, for $R\in (0,1]$ fixed, we consider the new function
\begin{equation}\label{de:vR}
    v_R(x):=u_0(Rx)\quad\text{for $x\in \Sigma\cap B_2$.}
\end{equation}
From \eqref{eq:u0} and via a simple change of variables, one finds that $v_R$ is solution to
\begin{equation}\label{eq:vR}
    -\Delta_p v_R=f_R\quad\text{in $\Sigma\cap B_2$},\quad f_R(x):=R^p\,f_0(Rx),
\end{equation}
and $v_R$ satisfies the same boundary condition of $u_0$ on $\partial \Sigma\cap B_2$.

Since $\Sigma_{\frac12,1}$ is a set of class $C^{1,\beta}$, owing to boundary regularity for the $p$-Laplace equation  \cite[Theorems 1.2-1.3]{A26} (see also \cite{L88}), we have that 
\[
v_R\in C^{1,\alpha}\big(\overline \Sigma_{1/2,1}\big)
\]
for some $\alpha\in (0,1)$ depending on $N,p,\d,\beta,\mathcal{L}_{\Sigma_{1/2,1}}$, together with  the quantitative estimate
\begin{equation}\label{quant:est_vR}
\begin{split}
    \|\nabla v_R&\|_{C^{0,\alpha}\big(\Sigma_{1/2,1})}
    \\
    &\leq C_1\bigg(N,p,\d,\beta,\mathrm{diam}(\Sigma_{1/2,1}),\mathcal{L}_{\Sigma_{1/2,1}},\|\partial\Sigma_{1/2,1}\|_{C^{1,\beta}},\|v_R\|_{W^{1,p}(\Sigma_{1/4,3/2})}, \|f_R\|_{L^{N+\d}(\Sigma_{1/4,\,3/2})}  \bigg).
    \end{split}
\end{equation}

% On the other hand, the quantities $\mathrm{diam}(\Sigma_{1/2,1}),\mathcal{L}_{\Sigma_{1/2,1}},\|\partial\Sigma_{1/2,1}\|_{C^{1,1}}$ depend solely on $\mathrm{diam}(D), \mathcal{L}_D, \|\partial D\|_{C^{1,1}} $,  respectively. {\color{red} Aggiungere tale info nei preliminari}

Therefore, scaling back to $u_0$, and using the observations in Remark \ref{remark:lipconst}, from \eqref{quant:est_vR} we deduce that
\begin{equation}\label{bound:nabu0}
\begin{split}
    \big\|R\,\nabla u_0\big\|&_{L^\infty(\Sigma_{R/2,R})}+\big[R^{1+\alpha}\,\nabla u_0\big]_{C^{0,\alpha}(\Sigma_{R/2,R})}
    \\
    \leq C_1\bigg(&N,p,\d,\beta,\mathrm{diam}(D),\mathcal{L}_{D},\|\partial D\|_{C^{1,\beta}}, \frac{1}{R^N}\int_{\Sigma_{R/4,3R/2}} |u_0|^p\,dx,
    \\
    &  \quad\frac{1}{R^N}\int_{\Sigma_{R/4,3R/2}}R^p\,|\nabla u_0|^p\,dx, \frac{1}{R^N}\int_{\Sigma_{R/4,3R/2}} R^{(N+\d)p}|f_0|^{N+\d}\,dx \bigg),
\end{split}
\end{equation}
with H\"older exponent $\alpha=\alpha(N,p,\beta,\delta,\mathcal{L}_D)\in (0,1)$ which does not depend on $R\in (0,1]$.
We now need to bound (uniformly in $R\in (0,1]$)  the last three integrals in the right hand-side above.

By Lemma \ref{lemma:energia}-- see \eqref{tesi:linf}-- and \eqref{mis:sigma}, we have
\begin{equation}\label{first:int}
    \frac{1}{R^N}\int_{\Sigma_{R/4,3R/2}}|u_0|^p\,dx\leq C(N,\mathcal{L}_D)\,\|u_0\|^p_{L^\infty(\Sigma\cap B_{3/2})}\leq C'(N,p,\mathcal{L}_D,\delta)
\end{equation}
for all $R\in (0,1]$. Moreover, from the energy estimate of Lemma \ref{lemma:energia} coupled with \eqref{mis:sigma} and the fact that $R/4\leq |x|\leq 3/2 R$, we find
\begin{equation}\label{temp:fin}
     \frac{1}{R^N}\int_{\Sigma_{R/4,3R/2}}R^p\,|\nabla u_0|^p\,dx \leq C(N,p,\mathcal{L}_D)\,\mint_{\Sigma\cap B_{3/2 R}}|x|^p\,|\nabla u_0|^p\,dx\leq C'(N,p,\d,\mathcal{L}_D).
\end{equation}
Finally we estimate
\begin{equation}\label{temp:f}
    \begin{split}
        \frac{1}{R^N}\int_{\Sigma_{R/4,3R/2}}& R^{(N+\d)p}|f_0|^{N+\d}\,dx
        \\
        &\leq C(N,p,\delta)\,\int_{\Sigma\cap B_{3/2 R}} |x|^{N(p-1)+\d\,p}|f_0|^{N+\d}\,dx\stackrel{\eqref{bounds}}{\leq} C'(N,p,\delta).
    \end{split}
\end{equation}

Therefore, by coupling \eqref{bound:nabu0} with \eqref{first:int}- \eqref{temp:f}, and since $R/2\leq|x|\leq R$, we deduce
\begin{equation}\label{temp:fin1}
    \tfrac{1}{2}\,\big\|R\,\nabla u_0\big\|_{L^\infty(\Sigma_{R/2,R})}\leq \big\||x|\,\nabla u_0\big\|_{L^\infty(\Sigma_{R/2,R})}\leq C\Big(N,p,\d,\beta,\mathrm{diam}(D),\mathcal{L}_{D},\|\partial D\|_{C^{1,\beta}}\Big)
\end{equation}
and
\begin{equation}\label{temp:uso}
    R^{1+\alpha}|\nabla u_0(x)-\nabla u_0(y)|\leq C_2\,|x-y|^{\alpha}\quad\text{for all $x,y\in \Sigma_{R/2,R}$}
\end{equation}
with $C_2=C_2\big(N,p,\d,\beta,\mathrm{diam}(D),\mathcal{L}_{D},\|\partial D\|_{C^{1,\beta}}\big)>0$.

Since the estimate \eqref{temp:fin1} is valid for all $R\in (0,1]$, this implies
\begin{equation}\label{first:fin}
    \big\||x|\,\nabla u_0\big\|_{L^\infty(\Sigma\cap B_1)}\leq C_3\bigg(N,p,\d,\beta,\mathrm{diam}(D),\mathcal{L}_{D},\|\partial D\|_{C^{1,\beta}}\bigg).
\end{equation}
Going back to $u$ via \eqref{def:u0}, this proves \eqref{weight:xnablau}.

Next, for $x,y\in\Sigma_{R/2,R}$, we estimate
\begin{equation}\label{az}
    \begin{split}
        \Big| |x|^{1+\alpha}\nabla u_0(x)&-|y|^{1+\alpha}\nabla u_0(y)\Big|\leq |x|^{1+\alpha}\big| \nabla u_0(x)-\nabla u_0(y)\big|+\Big||x|^{1+\alpha}-|y|^{1+\alpha} \Big|\,|\nabla u_0(y)|
        \\
        &\leq R^{1+\alpha}\,\big| \nabla u_0(x)-\nabla u_0(y)\big|+C(\alpha)\,R\,|\nabla u_0(y)|\,|x-y|^\alpha
        \\
        &\stackrel{\eqref{temp:uso},\eqref{temp:fin1}}{\leq} (C_2+C(\alpha)\,C_3)|x-y|^\alpha,
    \end{split}
\end{equation}
where in the second inequality we used the elementary estimate $$\Big||x|^{1+\alpha}-|y|^{1+\alpha} \Big|\leq C(\alpha) R^\alpha\,|x-y|\leq C'(\alpha)\,R\,|x-y|^{\alpha}$$ 
which holds true for all $x,y\in B_R\setminus B_{R/2}$. Hence, \eqref{az} shows that
\begin{equation}\label{az1}
    \Big[|x|^{1+\alpha}\,\nabla u_0(\cdot) \Big]_{C^{0,\alpha}(\Sigma_{R/2,R})}\leq C_4\bigg(N,p,\d,\beta,\mathrm{diam}(D),\mathcal{L}_{D},\|\partial D\|_{C^{1,\beta}}\bigg)\quad \text{for all $R\in (0,1]$.}
\end{equation}
We now claim that \eqref{az1} implies the desired H\"older continuity in the whole conical sector $\Sigma \cap B_1$, that is
\begin{equation}\label{second:fin}
    \Big[|x|^{1+\alpha}\,\nabla u_0(\cdot) \Big]_{C^{0,\alpha}(\Sigma\cap B_1)}\leq C_5\bigg(N,p,\d,\beta,\mathrm{diam}(D),\mathcal{L}_{D},\|\partial D\|_{C^{1,\beta}}\bigg),
\end{equation}
which together the definition of $u_0$ in \eqref{def:u0} yields our desired result \eqref{weigh:holdnablau}, that is our thesis.

So let $x,y\in \Sigma\cap  B_1$, and assume without loss of generality that $x\neq y$ and $|x|\leq |y|$. We may also assume that $x\neq 0$, for otherwise  we have  $|y|^{1+\alpha}|\nabla u_0(y)|\leq C|y|^\alpha$ by \eqref{first:fin}.

Thus assume that $x\neq 0$, and let $\ell_x$ be the segment joining $x$  with the origin. Let
\begin{equation*}
    R_k:=2^{k}\,|x|,\quad k=0,1,2,\dots
\end{equation*}
and let us take $x_j\in\ell_x\cap\Sigma\cap\partial B_{R_j}$, $j=0,1,2\dots$, so that $x_0=x$, $x_j=2^{j}\,x$ and $|x_j|=R_j$ (such a choice is admissible as $\Sigma$ is a cone). 

Let $m$ be the integer such that $y\in B_{R_{m+1}}\setminus B_{R_{m}}$, and we may assume that $m\geq 1$ for otherwise the estimate $\big||x|^{1+\alpha}\nabla u_0(x)-|y|^{1+\alpha}\nabla u_0(y)\big|\leq C\,|x-y|^\alpha$ follows from \eqref{az1}.
Let us also define $y_m=\frac{|x_m|}{|y|}y$, so that $y_m$ lies on the line segment $\ell_y$ joining $y$ with the origin,  $|y|\geq R_m=|y_m|=|x_m|$, and we have that 
\[
|y_m-y|=|x_m|-|y|=|y|-R_m\leq R_{m+1}= 2^{m+1}|x|.
\] 
Let us now recall  the following elementary inequality 
\begin{equation}\label{elem:proven}
\bigg|\frac{x}{|x|}-\frac{y}{|y|} \bigg|\leq \frac{4|x-y|}{|x|+|y|}.
\end{equation}
To prove it, for $0<|x|\leq|y|$ observe that 
\[
\begin{split}
    \bigg|\frac{x}{|x|}-\frac{y}{|y|} \bigg|&= \bigg|\frac{|y|x-|x|y}{|x|\,|y|} \bigg|=\bigg| \frac{(|y|-|x|)x+|x|(x-y)}{|x|\,|y|}\bigg|
    \\
    &\leq \frac{||x|-|y||+|x-y|}{|y|}\leq \frac{2|x-y|}{|y|}\leq  \frac{4|x-y|}{|x|+|y|}
\end{split}
\]
 where in the last inequality we used $|y|=|y|/2+|y|/2\geq 1/2\big(|x|+|y|\big)$. Thus, by \eqref{elem:proven}, we have 
\[
|y_m-x_m|=2^{m}\,|x|\bigg|\frac{x}{|x|}-\frac{y}{|y|} \bigg|\leq 2^{m}\,|x|\,\frac{4|x-y|}{|y|}\leq 4|x-y|.
\]

By exploiting this piece of information, by repeatedly using the triangle inequality, \eqref{az1}, that $|x_j-x_{j+1}|=R_j=2^{j}\,|x|$, by setting $H(x):=|x|^{1+\alpha} \nabla u_0(x)$, we get
\begin{equation}\label{az2}
    \begin{split}
        \big|H(x)-H(y)\big|&\leq \sum_{j=0}^{m-1}\big|H(x_j)-H(x_{j+1})\big|+|H(x_m)-H(y_m)|+|H(y_m)-H(y)|
        \\
        &\leq C_4\bigg(\sum_{j=0}^{m-1} |x_j-x_{j+1}|^\alpha+|x_m-y_m|^\alpha+|y_m-y|^\alpha\bigg)
        \\
        &\leq C_4\,|x|^\alpha\,\sum_{j=0}^{m-1}2^{\alpha j}+4^\alpha\,C_4\,|x-y|^\alpha+C_4(2^{m+1}|x|)^\alpha
        \\
        &\leq C_4'(2^{m} |x| )^\alpha+4^\alpha\,C_4\,|x-y|^\alpha,
    \end{split}
\end{equation}
with $C_4'$ depending on the same quantities as $C_4$. We now claim that
\begin{equation}\label{left_toprove}
    2^m|x|\leq 4\,|x-y|,
\end{equation}
which together with \eqref{az2} yields the thesis \eqref{second:fin}. To prove \eqref{left_toprove}, first observe that, by the triangle inequality $2^{m}|x|\leq |y|\leq |x|+|x-y|$. We distinguish two cases: suppose first that $|x-y|\geq |x|$. Then 
 $2^m|x|\leq 2\,|x-y|$, which proves \eqref{left_toprove} in this case. 
We are left to study the cases $|x-y|<|x|$, which implies $2^m|x|\leq |y|< 2|x|$, that is $m=0$, in contradiction with our assumption $m\geq 1$. Equations \eqref{az2}-\eqref{left_toprove} yield \eqref{second:fin}, thus completing the proof.
%Next, let us consider the line segment joining the origin and $x$, and we take a finite sequence of points $\{x_1,\dots,x_{k}\}$ such that $x_j\in \Sigma\cap B_{R_0^j}\setminus B_{R_0^{j+1}}$ for all $j=1,\dots k$ (this is possible since $\Sigma$ is a cone). We also set $x=x_0$. In particular, we have the additivity property
%\begin{equation*}
 %   |x_0-x_k|=\sum_{j=0}^{k-1}|x_j-x_{j+1}|.
%\end{equation*}
\end{proof}

We now move onto the proof of Theorem \ref{thm:weightH1}. First, we show the following
\begin{lemma}\label{lemma:H1annuli}
    Let $D, \Sigma$ be as in Theorem \ref{thm:weightH1}, and let $f\in L^2\cap L^{(p^{*})'}(\Sigma\cap B_2)$. Assume that $v\in W^{1,p}(\Sigma\cap B_2)$ is solution to either  \eqref{eq:bdryNeu} or \eqref{eq:bdryDir}. Then $\A(\nabla v)\in W^{1,2}(\Sigma\cap (B_{1}\setminus B_{1/2}))$, and there exists a constant $C_0$ depending on $N,p,\mathrm{diam}(D),\mathcal{L}_D, \|\partial D\|_{C^{1,1}}$ such that
    \begin{equation}\label{est:annuliH1}
        \|\nabla\A(\nabla v)\|_{L^{2}\big(\Sigma_{1/2,1})}\leq C_0\,\bigg\{\| \A(\nabla v)\|_{L^1\big(\Sigma_{1/4,3/2}\big) }+\|f\|_{L^2\big(\Sigma_{1/4,3/2}\big)} \bigg\}.
    \end{equation}
\end{lemma}

\begin{proof}
    The result is nothing else than a localization of the results in \cite{cia, ACCFM25}. For the sake of completeness, we provide the details of the proof.
    \\
\noindent\textit{Step 1: regularization.}
 Let $f_k\in C^\infty_c(\R^N)$ be a regularization of $f$ (say, via convolution) such that
    \begin{equation}\label{conv:fk}
        f_k\xrightarrow{k\to\infty }f\quad\text{in $L^2\cap L^{(p^*)'}(\Sigma\cap B_{7/4}\big)$.}
    \end{equation}
Next, recalling \eqref{regul:stress}, we consider $v_{\e,k}\in W^{1,p}( \Sigma_{1/4,3/2})$ the unique weak solution to 
\begin{equation}\label{eq:vekneu}
    \begin{cases}
        -\mathrm{div}\big(\A_\e(\nabla v_{\e,k}) \big)=f_k\quad &\text{in $\Sigma_{1/4,\,3/2}$}
        \\
        \partial_\nu v_{\e,k}=0\quad &\text{on $\partial\Sigma_{1/4,\,3/2}$}
        \\
        v_{\e,k}=v \quad &\text{on $\Sigma\cap \partial (B_{3/2}\setminus B_{1/4})$.}
    \end{cases}
\end{equation}
 if $v$ solves the Neumann problem \eqref{eq:bdryNeu}, or the weak solution to

\begin{equation}\label{eq:vekdir}
    \begin{cases}
        -\mathrm{div}\big(\A_\e(\nabla v_{\e,k}) \big)=f_k\quad &\text{in $\Sigma_{1/4,\,3/2}$}
        \\
         v_{\e,k}=0\quad &\text{on $\partial\Sigma_{1/4,\,3/2}$}
        \\
        v_{\e,k}=v \quad &\text{on $\Sigma\cap \partial (B_{3/2}\setminus B_{1/4})$,}
    \end{cases}
\end{equation}
if $v$ solves to the Dirichlet problem \eqref{eq:bdryDir}.

 Let us start by deriving an energy estimate for $v_{\e,k}$; by testing \eqref{eq:vekneu}-\eqref{eq:vekdir} with $v_{\e,k}-v$, we get (omitting the set of integration $\Sigma_{1/4,\,3/2}$)
\begin{equation}\label{eps:noia}
\begin{split}
\int\Big\{ &\big[\e^2+|\nabla v_{\e,k}|^2\big]^{\frac{p-2}{2}}\,|\nabla v_{\e,k}|^2-\big[\e^2+|\nabla v_{\e,k}|^2\big]^{\frac{p-2}{2}}\,|\nabla v_{\e,k}|\,|\nabla v|\Big\}\,dx
\\
&\leq \int \A_\e(\nabla v_{\e,k})\cdot \nabla (v_{\e,k}-v)\,dx=\int f_k\,(v_{\e,k}-v)\,dx
\\
&\leq \|f\|_{L^{(p^*)'}}\,\Big(\int |v_{\e,k}-v|^{p^*}\,dx\Big)^{1/(p^*)}
\\
&\leq C_1\,\|f\|_{L^{(p^*)'}}\,\Big(\int |\nabla (v_{\e,k}-v)|^p\,dx\Big)^{1/p}
\\
&\leq C_2\,\gamma^{-p'}\,\|f\|^{p'}_{L^{(p^*)'}}+\gamma^p\,\Big(\int |\nabla v_{\e,k}|^p+|\nabla v|^p\,dx\Big)
\end{split}
\end{equation}
for all $\gamma\in (0,1)$, with $C_1,C_2$ depending on $N,p,\delta, \mathrm{diam}(D), \mathcal{L}_D$. Above, in the last three estimates we used H\"older's, Young's, \eqref{conv:fk} and Sobolev inequality, whose quantitative constant depends on the Lipschitz characteristics $\mathcal{L}_D$ in view of Remark \ref{remark:lipconst}.

We now recall the following estimate
\[
\big[\e^2+|\nabla v_{\e,k}|^2\big]^{\frac{p-2}{2}}\,|\nabla v_{\e,k}|\,|\nabla v|\leq \frac{1}{2}\big[\e^2+|\nabla v_{\e,k}|^2\big]^{\frac{p-2}{2}}\,|\nabla v_{\e,k}|^2+C(p)\,\big[\e^2+|\nabla v|^2\big]^{\frac{p-2}{2}}\,|\nabla v|^2,
\]
which is a consequence of Young's inequality \cite[Eq. (2.26)]{A26} applied with $a_\e(t),b_\e(t),B_\e(t)$ as in \eqref{def:aebeBe}. The constant $C(p)$ is independent of $\e$ thanks to \eqref{unif:ineps}.

Combining the above displayed equation with \eqref{eps:noia}, we get
\begin{equation}\label{eps:noia1}
\begin{split}
    \int \big[\e^2+|\nabla v_{\e,k}|^2\big]^{\frac{p-2}{2}}\,|\nabla v_{\e,k}|^2\,dx\leq C\,\gamma^{-p'}\,\|f\|_{L^{(p^*)'}}+\gamma^p\,\int |\nabla v_{\e,k}|^p+C\,\int\big(1+|\nabla v|^p\big)\,dx.
    \end{split}
\end{equation}
Moreover, in the case $p\geq 2$, we have $\big[\e^2+t^2\big]^{\frac{p-2}{2}}t^2\geq t^p$, whereas in the case  $p<2$ there holds
\[
\big[\e^2+t^2\big]^{\frac{p-2}{2}}t^2\geq
2^{\frac{p-2}{2}}\,t^p\quad \text{if }\,t\geq 1.
\]
Thus, in both cases we have the estimate 
\[
\int |\nabla v_{\e,k}|^p\,dx\leq \int 1\,dx+\int_{\{|\nabla v_{\e,k}|>1\}}\big[\e^2+|\nabla v_{\e,k}|^2\big]^{\frac{p-2}{2}}|\nabla v_{\e,k}|^2\,dx,
\]
which combined with \eqref{eps:noia1}, and via simple algebraic manipulation yields 
\begin{equation}\label{energy:vek}
    \int (1+|\nabla v_{\e,k}|^p)\,dx\leq C\,\|f\|_{L^{(p^*)'}}+C\,\int\big(1+|\nabla v|^p\big)\,dx,
\end{equation}
with constant $C>0$ independent of $\e,k$.
Then, by Poincar\'e inequality we have
\begin{equation*}
    \int |v_{\e,k}-v|^p\,dx\leq C\, \int |\nabla (v_{\e,k}-v)|^p\,dx,
\end{equation*}
 which together with \eqref{energy:vek} implies
\begin{equation}\label{en:vek1}
    \int (1+| v_{\e,k}|^p)\,dx\leq C\,\|f\|_{L^{(p^*)'}}+C\,\int\big(1+|\nabla v|^p\big)\,dx,
\end{equation}
with $C>0$ independent of $\e,k$.
 
Next, by regularity theory for $p$-Laplace problems \cite{A26} combined with \eqref{energy:vek}, \eqref{en:vek1}, we have that  $v_{\e,k}\in C^{1,\alpha}(\overline\Sigma_{t/4,\,3t/2})$ for every $t\in (0,1)$, and for some $\alpha_k\in (0,1)$ independent of $\e$, with quantitative estimate 
\begin{equation}\label{stima:tempC1ak}
     \|v_{\e,k}\|_{C^{1,\alpha_k}(\overline\Sigma_{t/4,\,3t/2})}\leq C_{k,t},
 \end{equation}
 with $C_{k,t}$ possibly depending on $k,t$, but independent of $\e$.
 
 Moreover, being $\nabla v_{\e,k}$ bounded, Equations \eqref{eq:vekneu}-\eqref{eq:vekdir} are uniformly elliptic, hence by standard elliptic regularity theory \cite[Theorem 9.19]{GT}, \cite[Section 4.7]{L13}  we deduce that 
 \[
 v_{\e,k}\in C^{2,\beta}\big(\overline\Sigma_{t/4,\,3t/2}\big),\quad \text{for every $t\in (0,1)$.}
 \]

\noindent \textit{Step 2: estimate for the regularized Stress field.}
Now, in the case of Dirichlet problems we  make use of \cite[Eq. (4.35)]{ACCFM25} with $H=\text{Euclidean norm}$,  $v=v_{\e,k}$,    $h=a_\e\big(|\nabla v_{\e,k}|\big)$, while in the case of Neumann problems we exploit  \cite[Eq. (8.15)]{ACCFM25} (see also \cite[Theorem 3.1.1.1]{gris}), and we couple the resulting equations with \eqref{elm:trace} so that, in both cases, we find 
\begin{equation}\label{ineq:fund}
\begin{split}
   \hat C(N,p)\,\int |\nabla \A_\e(\nabla v_{\e,k})|^2\,\phi^2\,dx
   \leq &\, \int |\mathrm{div}\big(\A_\e(\nabla v_{\e,k})\big)|^2\,\phi^2\,dx
\\
&+\int\mathrm{div}\big(\A_\e(\nabla v_{\e,k})\big)\,\A_\e(\nabla v_{\e,k})\cdot \nabla \phi^2\,dx
   \\
   &-\int \nabla \A_\e(\nabla v_{\e,k})\,\A_\e(\nabla v_{\e,k})\cdot \nabla \phi^2\,dx+ I_{\partial \Sigma}
   \end{split}
\end{equation}
for all $\phi\in C^\infty_c(B_{3/2}\setminus B_{1/4})$. Above, we set $I_{\partial \Sigma}=0$ if $\mathrm{supp}\,\phi\cap \partial\Sigma=\emptyset$, otherwise
\begin{equation*}
    I_{\partial \Sigma}:=\begin{cases}
       \int_{\partial \Sigma} |\A_\e(\nabla v_{\e,k})|^2\,\,\mathrm{tr}\mathcal{B}\,\phi^2\,d\mathcal{H}^{N-1}\quad&\text{in the case of Dirichlet problems \eqref{eq:bdryDir}}
       \\
       \\
       \int_{\partial \Sigma} \,\mathcal{B}\Big( \A_\e(\nabla v_{\e,k})_T,\A_\e(\nabla v_{\e,k})_T\Big)\,\phi^2\,d\mathcal{H}^{N-1}\quad&\text{in the case of Neumann problems \eqref{eq:bdryNeu}},
   \end{cases}
\end{equation*}
where $\mathcal{B}$ is the second fundamental form of $\partial \Sigma$, $\mathrm{tr}\mathcal{B}$ is the mean curvature, and  $\A_\e(\nabla v_{\e,k})_T$ is the tangential part (w.r.t. $\partial \Sigma$) of the vector $\A_\e(\nabla v_{\e,k})$.
By exploiting the trace inequality \cite[Proposition 6.2 \& Corollary 6.6]{ACCFM25}, for all $\phi\in C^\infty_c(B_R(x_0))$, with $x_0\in \partial \Sigma$, we have the estimate
\begin{equation}\label{est:termbord}
\begin{split}
    |I_{\partial \Sigma}|&\leq \int_{\partial \Sigma}|\A_\e(\nabla v_{\e,k})|^2\,|\mathcal{B}|\,\phi^2\,d\mathcal{H}^{N-1}\leq h(R)\,\int \big|\nabla [\A_\e(\nabla v_{\e,k})\,\phi]\big|^2\,dx
    \\
    &=h(R)\,\int |\nabla \A_\e(\nabla v_{\e,k})|^2\,\phi^2\,dx+h(R)\,\int | \A_\e(\nabla v_{\e,k})|^2\,|\nabla \phi|^2\,dx,
    \end{split}
\end{equation}
where, recalling Remark \ref{remark:lipconst}, we set 
\[
h(R):=
\begin{cases}
   C(N,\mathcal{L}_D)\,R\,\|\mathcal{B}\|_{L^\infty(\partial\Sigma_{1/4,\,3/2})}\quad &\text{if $N\geq 3$}
    \\
    \\C(N,\mathcal{L}_D)\,R\,\log(1+1/R)\,\|\mathcal{B}\|_{L^\infty(\partial\Sigma_{1/4,\,3/2})}\quad &\text{if $N=2$.}
\end{cases}
\]
Therefore, by fixing a radius $R_0=R_0(N,p,\mathcal{L}_D\,\|\partial D\|_{C^{1,1}})\in (0,1)$ \footnote{Here $\underline{C}=\underline{C}(N,L_{\Sigma_{1/4,3/2}})=\underline{C}(N,\mathcal{L}_D)$ is the  constant appearing in Lemma \ref{lemm:bdr:sobannuli}} such that
\begin{equation}\label{fix:R0}
    h(R_0)\leq \frac{\hat{C}(N,p)}{4},\quad  R_0\leq \frac{R_{\Sigma_{1/4,\,3/2}}}{4\underline{C}(1+L_{\Sigma_{1/4,\,3/2}})};
\end{equation}
then by combining \eqref{ineq:fund} -\eqref{est:termbord}, and using Equations \eqref{eq:vekneu}-\eqref{eq:vekdir}, we find
\begin{equation}\label{into}
    \begin{split}
   \frac{3\,\hat C(N,p)}{4}\,\int |\nabla \A_\e(\nabla v_{\e,k})|^2\,\phi^2\,dx
   \leq &\, \int |f_k|^2\,\phi^2\,dx+C(N)\,\int\big|\nabla \big(\A_\e(\nabla v_{\e,k})\big)\big|\,\big|\A_\e(\nabla v_{\e,k})\big|\,|\nabla \phi^2|\,dx
   \\
   &+C(N,p)\,\int | \A_\e(\nabla v_{\e,k})|^2\,|\nabla \phi|^2\,dx,
   \end{split}
\end{equation}
for all $\phi\in C^\infty_c(B_{3/2}\setminus B_{1/4})$ such that either $\mathrm{supp}\phi\cap \partial\Sigma=\emptyset$, or such that $\phi\in C^\infty_c(B_R(x_0))$, $x_0\in \partial \Sigma_{1/4,\,3/2}$, with  $R\leq R_0$. By weighted Young's inequality
\[
\begin{split}
C(N)\,\int\big|\nabla \big(\A_\e(\nabla v_{\e,k})&\big)\big|\,\big|\A_\e(\nabla v_{\e,k})\big|\,|\nabla \phi^2|\,dx
\\
&\leq \frac{1}{4}\int |\nabla \A_\e(\nabla v_{\e,k})|^2\,\phi^2\,dx+C(N)\,\int |\A_\e(\nabla v_{\e,k})|^2\,|\nabla\phi|^2\,dx,
\end{split}
\]
which inserted into \eqref{into} gives
\begin{equation}\label{start:covering}
    \begin{split}
         \int |\nabla \A_\e(\nabla v_{\e,k})|^2\,\phi^2\,dx
   \leq &\, C(N,p)\,\int |f_k|^2\,\phi^2\,dx+C(N,p)\,\int | \A_\e(\nabla v_{\e,k})|^2\,|\nabla \phi|^2\,dx.
    \end{split}
\end{equation}
Now let $x_0\in \partial \Sigma_{1/4,\,3/2}$, $R_0/2\leq s<r\leq R_0$, and consider a cut-off function $\phi_{r,s}\in C^\infty_c(\mathcal{B}_{r}(x_0))$, $\phi_{r,s}\equiv 1$ on $\mathcal{B}_{s}(x_0)$, and $|\nabla \phi_{r,s}|\leq C(N)/(r-s)$, where $\mathcal{B}_r(x_0)$ is defined by  \eqref{flatten:balls}.
Using \eqref{start:covering} with $\phi=\phi_{r,s}$, we get
\begin{equation}\label{annuli1}
\begin{split}
     \int_{\Sigma\cap \mathcal{B}_s(x_0)} &|\nabla \A_\e(\nabla v_{\e,k})|^2\,dx
   \\
   &\leq \, C(N,p)\,\int_{\Sigma\cap \mathcal{B}_{R_0}(x_0)} |f_k|^2\,dx+\frac{C(N,p)}{(r-s)^2}\,\int_{\Sigma\cap \big(\mathcal{B}_r(x_0)\setminus \mathcal{B}_s(x_0)\big)} | \A_\e(\nabla v_{\e,k})|^2\,dx,
   \end{split}
\end{equation}
while by Lemma \ref{lemm:bdr:sobannuli} with $R=R_0$ and $\epsilon=(r-s)\,R_0^{-1}$, we get 
\begin{equation}\label{annuli2}
\begin{split}
   & \frac{C(N,p)}{(r-s)^2}\int_{\Sigma\cap \big(\mathcal{B}_r(x_0)\setminus \mathcal{B}_s(x_0)\big)} | \A_\e(\nabla v_{\e,k})|^2\,dx
    \\
    &\leq C_1\,\int_{\Sigma\cap \big(\mathcal{B}_r(x_0)\setminus \mathcal{B}_s(x_0)\big)} |\nabla \A_\e(\nabla v_{\e,k})|^2\,dx+\frac{C_1\,R_0}{(r-s)^{N+3}}\bigg(\int_{\Sigma\cap \big(\mathcal{B}_r(x_0)\setminus \mathcal{B}_s(x_0)\big)} | \A_\e(\nabla v_{\e,k})|\,dx \bigg)^2
    \end{split}
\end{equation}
with $C_1=C_1(N,p,\mathcal{L}_D)$. Coupling \eqref{annuli1}-\eqref{annuli2} tells that
\begin{equation*}
    \begin{split}
         \int_{\Sigma\cap \mathcal{B}_s(x_0)} &|\nabla \A_\e(\nabla v_{\e,k})|^2\,dx
   \\
   \leq&\, C_1\,\int_{\Sigma\cap \big(\mathcal{B}_r(x_0)\setminus \mathcal{B}_s(x_0)\big)} |\nabla \A_\e(\nabla v_{\e,k})|^2\,dx+ C(N,p)\,\int_{\Sigma\cap \mathcal{B}_{R_0}(x_0)} |f_k|^2\,dx
   \\
   &+\frac{C_1\,R_0}{(r-s)^{N+3}}\bigg(\int_{\Sigma\cap \big(\mathcal{B}_r(x_0)\setminus \mathcal{B}_s(x_0)\big)} | \A_\e(\nabla v_{\e,k})|\,dx \bigg)^2,
    \end{split}
\end{equation*}
so adding the quantity $C_1\int_{\Sigma\cap \mathcal{B}_s(x_0)} |\nabla \A_\e(\nabla v_{\e,k})|^2\,dx $  to both sides of the equation above gives
\begin{equation}\label{annuli3}
        \begin{split}
         \int_{\Sigma\cap \mathcal{B}_s(x_0)} &|\nabla \A_\e(\nabla v_{\e,k})|^2\,dx
   \\
   \leq&\, \frac{C_1}{1+C_1}\,\int_{\Sigma\cap \mathcal{B}_r(x_0)} |\nabla \A_\e(\nabla v_{\e,k})|^2\,dx+ C(N,p)\,\int_{\Sigma\cap \mathcal{B}_{R_0}(x_0)} |f_k|^2\,dx
   \\
   &+\frac{C_1\,R_0}{(r-s)^{N+3}}\bigg(\int_{\Sigma\cap \big(\mathcal{B}_{R_0}(x_0)\setminus \mathcal{B}_{R_0/2}(x_0)\big)} | \A_\e(\nabla v_{\e,k})|\,dx \bigg)^2,
    \end{split}
\end{equation}
for all $R_0/2<s<r<R_0$. The iteration lemma \cite[Lemma 6.1, pp. 191]{G03} allows us to deduce
\begin{equation}\label{mid:covering}
            \begin{split}
         \int_{\Sigma\cap \mathcal{B}_{R_0/2}(x_0)} &|\nabla \A_\e(\nabla v_{\e,k})|^2\,dx
   \\
   \leq&\, C_2\,\int_{\Sigma\cap \mathcal{B}_{R_0}(x_0)} |f_k|^2\,dx+\frac{C_2\,}{R_0^{N+2}}\bigg(\int_{\Sigma\cap \big(\mathcal{B}_{R_0}(x_0)\setminus \mathcal{B}_{R_0/2}(x_0)\big)} | \A_\e(\nabla v_{\e,k})|\,dx \bigg)^2,
    \end{split}
\end{equation}
with $C_2$ depending on the same quantities as $C_1$. Equation \eqref{mid:covering} can be obtained also for balls $B_{R_0}(x_0)\Subset \Sigma_{1/4,\,3/2}$, repeating the very same argument from \eqref{annuli1}-\eqref{mid:covering}, using Lemma \ref{lem:sobannuli}. Alternatively, one can appeal to the interior regularity result of \cite[Theorem 2.1]{cia} or \cite[Theorem 2.1]{ACCFM25}.

Then, recalling the dependency on the data of $R_0$ in \eqref{fix:R0}, Equation \eqref{mid:covering} and a standard covering argument allow one to get the desired estimate on the regularized stress field, namely
\begin{equation}\label{estimate:forregularized}
        \begin{split}
         \int_{\Sigma_{1/2,\,1}} &|\nabla \A_\e(\nabla v_{\e,k})|^2\,dx
   \\
   \leq&\, C_3\,\int_{\Sigma_{1/4,\,3/2}} |f_k|^2\,dx+C_3\,\bigg(\int_{\Sigma_{3/8,5/4}} | \A_\e(\nabla v_{\e,k})|\,dx \bigg)^2,
    \end{split}
\end{equation}
with $C_3$ depending on the same quantities as $C_1,C_2$.
\\

\noindent\textit{Step 3. Passing to the limit $\e\to 0$.} Owing to \eqref{energy:vek}-\eqref{en:vek1} and \eqref{stima:tempC1ak}, we may extract a subsequence, still labeled as $v_{\e,k}$ such that, for all $k\in \N$,
\begin{equation}\label{first:convergence}
    \begin{split}v_{\e,k}&\xrightarrow{\e\to 0^+} w_k
    \\
    &\quad\text{weakly in $W^{1,p}(\Sigma_{1/4,\,3/2})$, and strongly in $C^1\big(\overline{\Sigma}_{t/4,\,3t/2}\big)$, for all $t\in (0,1)$,}
    \end{split}
\end{equation}
for some function $w_k$. Thereby passing to the limit in \eqref{energy:vek}-\eqref{en:vek1}, we deduce
\begin{equation}\label{wk:stimaen}
    \|w_k\|_{W^{1,p}(\Sigma_{1/4,\,3/2})}\leq C\,\big(1+\|f\|_{L^{(p^*)'}(\Sigma_{1/4,\,3/2})}\big)^{1/p}+C\,\|v\|_{W^{1,p}(\Sigma_{1/4,\,3/2})},
\end{equation}
with constant $C>0$ independent of $k$.

Then, from \eqref{first:convergence}, and since $\A_\e(\xi)\xrightarrow{\e\to 0^+}\A(\xi)$ locally uniformly in $\R^N$,  we get
\begin{equation}\label{conv:Aewk}
    \A_\e(\nabla v_{\e,k})\xrightarrow{\e\to 0^+} \A(\nabla w_k)\quad\text{uniformly in $\overline{\Sigma}_{t/4,3t/2}$, for all $t\in (0,1)$,}
\end{equation}
while thanks to  \eqref{estimate:forregularized}, up to subsequences
\begin{equation*}
    \A_\e(\nabla v_{\e,k})\xrightarrow{\e\to 0^+}\A(\nabla w_k)\quad\text{weakly in $W^{1,2}(\Sigma_{1/2,1})$,}
\end{equation*}
with estimate
\begin{equation}\label{second:estimate}
            \begin{split}
         \int_{\Sigma_{1/2,\,1}} &|\nabla \A(\nabla w_{k})|^2\,dx
   \\
   \leq&\, C_3\,\int_{\Sigma_{1/4,\,3/2}} |f_k|^2\,dx+C_3\,\bigg(\int_{\Sigma_{3/8,5/4}} | \A(\nabla w_{k})|\,dx \bigg)^2.
    \end{split}
\end{equation}

Furthermore, testing the weak formulations of \eqref{eq:vekdir}-\eqref{eq:vekneu}, and using \eqref{first:convergence}, \eqref{conv:Aewk} and the continuity of the trace operator, by passing to the limit $\e\to 0^+$ it is immediate to see that $w_k$ is a weak solution to either
\begin{equation}\label{eq:wkneu}
    \begin{cases}
        -\mathrm{div}\big(\A(\nabla w_{k}) \big)=f_k\quad &\text{in $\Sigma_{1/4,\,3/2}$}
        \\
        \partial_\nu w_{k}=0\quad &\text{on $\partial\Sigma_{1/4,\,3/2}$}
        \\
        w_{k}=v \quad &\text{on $\Sigma\cap \partial (B_{3/2}\setminus B_{1/4})$,}
    \end{cases}
\end{equation}
when $v$ is solution to the Neumann problem \eqref{eq:bdryNeu}, or solution to
\begin{equation}\label{eq:wkdir}
    \begin{cases}
        -\mathrm{div}\big(\A(\nabla w_{k}) \big)=f_k\quad &\text{in $\Sigma_{1/4,\,3/2}$}
        \\
         w_{k}=0\quad &\text{on $\partial\Sigma_{1/4,\,3/2}$}
        \\
        w_{k}=v \quad &\text{on $\Sigma\cap \partial (B_{3/2}\setminus B_{1/4})$,}
    \end{cases}
\end{equation}
if $v$ is solution to the Dirichlet problem \eqref{eq:bdryDir}.
\\

\textit{Step 4. Conclusion.} We are left to pass to the limit as $k\to \infty$. To this end, we observe that, from the energy estimates \eqref{wk:stimaen}, \eqref{second:estimate} and \eqref{conv:fk}, we have (up to nonrelabeled subsequences)
\begin{equation}\label{weak:quasifinish}
    \begin{split}
        &w_k\to w\quad\text{weakly in $W^{1,p}(\Sigma_{1/4,3/2})$}
        \\
        &\A(\nabla w_k)\to \bar\A\quad \text{a.e. and weakly in $W^{1,2}(\Sigma_{1/2,1})$,}
    \end{split}
\end{equation}
for some function $w$ and a vector function $\bar{\A}$.

We now claim that
\begin{equation}\label{puntuale:nwk}
    \nabla w_k(x)\to\nabla w(x)\quad\text{for a.e. $x\in \Sigma_{1/4,3/2}$.}
\end{equation}
Exploiting either \eqref{eq:wkneu} or \eqref{eq:wkdir}, the proof of \eqref{puntuale:nwk} is exactly the same as \cite[Eq. (4.58), pp. 128-129]{CM11} (see also \cite[proof of Eq. (6.82)]{CM14}, \cite[pp. 28-29]{AC251}, \cite[pp. 28-29]{ACP25}). The only difference is that we have a different integrability assumption on $f_k$, so we just need to modify \cite[Eq. (4.61)]{CM11}. Specifically, we set
\[
\vartheta=\inf\{[\A(\xi)-\A(\eta)]\cdot [\xi-\eta]:\,|\xi-\eta|>t,|\xi|\leq \tau,\,|\eta|\leq \tau  \},
\]
which is positive by the coercivity properties of $\A(\xi)$; for $k,l\in \N$, we make use of $w_k-w_l$ as a test function in the weak formulation of either \eqref{eq:wkneu}-\eqref{eq:wkdir}-and of its analogue for $w_l$, and subtracting the resulting equations yields
\begin{equation*}\label{modify}
\begin{split}
        \vartheta & \,|\{|\nabla w_k-\nabla w_l|>t,|\nabla w_k|\leq \tau,\,|\nabla w_l|\leq \tau \}|\leq \int[\A(\nabla w_k)-\A(\nabla w_l)]\cdot [\nabla w_k-\nabla w_l]\,dx
    \\
    &=\int (f_k-f_l)\,(w_k-w_l)\,dx\leq \|f_k-f_l\|_{L^{(p^*)'}}\,\|w_k-w_l\|_{L^{p^*}}
    \\
    &\leq C\,\|f_k-f_l\|_{L^{(p^*)'}}\,\|w_k-w_l\|_{W^{1,p}}\leq C'\,\|f_k-f_l\|_{L^{(p^*)'}}
    \end{split}
\end{equation*}
where in the last two estimates we used H\"older's, Sobolev inequality and \eqref{wk:stimaen} (above, all the integrals are evaluated at $\Sigma_{1/4,3/2}$). The last term above goes to zero as $k,l\to\infty$ by \eqref{conv:fk}. The remainder of the argument in \cite{CM11}, showing that \(\{\nabla w_k\}_{k\in\N}\) is a Cauchy sequence in measure, then applies verbatim, and we therefore omit it. In particular, up to nonrelabeled subsequences, this yields the convergence in \eqref{puntuale:nwk}, which also implies that  $\bar \A=\A(\nabla w)$.

Moreover, by letting $k\to\infty$ in the weak formulation of \eqref{eq:wkdir}-\eqref{eq:wkneu}, and exploiting \eqref{weak:quasifinish}-\eqref{puntuale:nwk} we infer that $w$ is a weak solution to either
\begin{equation*}
    \begin{cases}
        -\mathrm{div}\big(\A(\nabla w) \big)=f\quad &\text{in $\Sigma_{1/4,\,3/2}$}
        \\
         \partial_\nu w=0\quad\text{or}\quad w=0\quad &\text{on $\partial\Sigma_{1/4,\,3/2}$}
        \\
        w=v \quad &\text{on $\Sigma\cap \partial (B_{3/2}\setminus B_{1/4})$,}
    \end{cases}
\end{equation*}
according to whether  $v$ is solution to the Dirichlet  or Neumann problem, respectively. By uniqueness, this implies that $w=v$. Thereby using \eqref{weak:quasifinish} and the lower semicontinuity of the norm, letting $k\to\infty$ in  \eqref{second:estimate} gives \eqref{est:annuliH1}, that is our thesis.
\end{proof}

We are now in the position to prove Theorem \ref{thm:weightH1}.
\begin{proof}[Proof of Theorem \ref{thm:weightH1}]
    The interior $W^{1,2}$-regularity of $\A(\nabla u)$ is the content of \cite{ACCFM25,CM11}. We now apply a scaling procedure as in the proof of Theorem \ref{thm:nonlingrad}. So let $u_0, f_0$ be given by \eqref{def:u0}, with 
\begin{equation}\label{new:fnorm}
\|f\|:=\|f\|_{L^{\max\{N/p,1\}+\delta}(\Sigma\cap B_2)}+\|f\|_{L^2(\Sigma\cap B_2;\,|x|^{2p}dx)}.
\end{equation}

    Then for $R\in (0,1]$, let $v_R,f_R$ be the functions defined in \eqref{de:vR}, so that $v_R$  is solution to \eqref{eq:vR}, with the same boundary condition as $u$. Noticing that $f\in L^{(p^*)'}$ as $\max\{N/p,1\}+\delta>(p^*)'$, we may appeal to Lemma \ref{lemma:H1annuli} and find
    \begin{equation*}
        \int_{\Sigma_{1/2,1}}|\nabla \A(\nabla v_R)|^2\,dx\leq C\,\bigg(\int_{\Sigma_{1/4,3/2}}| \A(\nabla v_R)|\,dx\bigg)^2+C\,\int_{\Sigma_{1/4,3/2}}|f_R|^2\,dx,
    \end{equation*}
with $C>0$ depending on $N,p,\mathrm{diam}(D),\mathcal{L}_D,\|\partial D\|_{C^{1,1}}$, but independent of $R$.

Scaling back to $u_0$ and $f_0$ yields
\begin{equation*}
    \int_{\Sigma_{R/2,R}}R^{2p}\,|\nabla \A(\nabla u_0)|^2\,dx\leq C\,R^N\,\bigg(\mint_{\Sigma_{R/4,3R/2}}R^{p-1} |\nabla u_0|^{p-1}\,dx\bigg)^2+C\,\int_{\Sigma_{R/4,3R/2}}R^{2p}\,|f_0|^2\,dx.
\end{equation*}
    By H\"older's inequality and \eqref{tesi:energia}, we have
    \begin{equation*}
       \mint_{\Sigma_{R/4,3R/2}}R^{p-1} |\nabla u_0|^{p-1}\,dx\leq C\, \bigg(\mint_{\Sigma_{R/4,3R/2}}|x|^{p} |\nabla u_0|^{p}\,dx\bigg)^{1/p'}\leq C'(N,p,\delta, L_D).
    \end{equation*}
Merging the two estimates above gives
\begin{equation}\label{to:iterate}
    \int_{\Sigma_{R/2,R}}|x|^{2p}\,|\nabla \A(\nabla u_0)|^2\,dx\leq C\,R^N+C\,\int_{\Sigma_{R/4,3R/2}}|x|^{2p}\,|f_0|^2\,dx.
\end{equation}
In such estimate, we take $R=R_k=1/2^k$ for $k=0,1,\dots$, and summing over $k$ yields
\begin{equation}\label{iterared}
    \int_{\Sigma\cap B_1}|x|^{2p}\,|\nabla \A(\nabla u_0)|^2\,dx\leq C+C\,\int_{\Sigma\cap B_2}|x|^{2p}\,|f_0|^2\,dx,
\end{equation}
where we also used that
\begin{equation*}
    \sum_{k=0}^\infty \int_{\Sigma_{2^{-k-2},\,2^{-k+1}}}|x|^{2p}|f|^2\,dx\leq 3\,\int_{\Sigma\cap B_2} |x|^{2p}|f|^2\,dx,
\end{equation*}
since along the summation at most 3 annuli intersect.
Finally, rewriting \eqref{iterared} in terms of $u$ gives our thesis.
    
\end{proof}

\section{Proof of Theorem \ref{thm}}\label{sec:laplneu}

Let $D\subsetneq\mathbb S^{N-1}$ be a $C^2$ domain  of the unit sphere in $\mathbb R^N$, $N\geq 2$, and let $\Sigma_D$ be given by \eqref{def:cone}.
% $$
% \Sigma_D=\{x=r\theta:r>0,\theta\in D\}.
% $$
Consider the  Neumann eigenvalue problem 
\begin{equation}\label{eig_lap_N}
\begin{cases}
-\Delta_{\mathbb S^{N-1}}u=\lambda u & {\rm in\ }D\\
\partial_{\nu}u=0 & {\rm on\ }\partial D,
\end{cases}
\end{equation}
where $\Delta_{\mathbb S^{N-1}}$ is the Laplace-Beltrami operator on $\mathbb S^{N-1}$. It is well-known that problem \eqref{eig_lap_N} admits an increasing sequence of non-negative eigenvalues of finite multiplicity:
$$
0=\lambda_0<\lambda_1\leq\lambda_2\leq\cdots\leq\lambda_j\leq\cdots\nearrow+\infty.
$$
and a corresponding orthonormal basis $\{Y_j\}_{j=0}^{\infty}$ of $L^2(D)$ of eigenfunctions. 
We have that $Y_0={\rm const}$. 

We also consider the eigenvalue problem for the Laplacian on $D$ with Dirichlet boundary conditions:
\begin{equation}\label{eig_lap_D}
\begin{cases}
-\Delta_{\mathbb S^{N-1}}u=\lambda u\,, & {\rm in\ }D\,,\\
u=0\,, & {\rm on\ }\partial D.
\end{cases}
\end{equation}
It is well-known that problem \eqref{eig_lap_D} admits an increasing sequence of positive eigenvalues of finite multiplicity:
$$
0<\lambda_1<\lambda_2\leq\cdots\leq\lambda_j\leq\cdots\nearrow+\infty.
$$
and a corresponding orthonormal basis $\{Y_j\}_{j=1}^{\infty}$ of $L^2(D)$ of eigenfunctions.

\smallskip

In what follows, with abuse of notation, we will write $\lambda_j$ and $Y_j$ for both the Neumann and Dirichlet eigenvalues and eigenfunctions, and we will denote by $L_D$ the (negative) Laplacian on $D$ with Neumann or Dirichlet boundary conditions. Hence we will say that $\lambda_j$ and $Y_j$ are the eigenvalues and eigenfunctions of $L_D$ to indicate that they are the eigenvalues and eigenfunctions of problems \eqref{eig_lap_N} and \eqref{eig_lap_D}.

\smallskip

For any $j\geq 1$ we denote by $\xi_j$ and $\xi_{-j}$ the positive and negative roots of $\xi(\xi+N-2)=\lambda_j(D)$, respectively:
$$
\xi_j=\frac{-(N-2)+\sqrt{4\lambda_j(D)+(N-2)^2}}{2}\ \ \ \ \ \ \xi_{-j}=\frac{-(N-2)-\sqrt{4\lambda_j(D)+(N-2)^2}}{2}.
$$
 For $j=0$, we employ the notation $\xi_0$ and $\xi_{-0}$, where we understand that, for $N=2$ $\xi_0=\xi_{-0}=0$, while for $N\geq 3$ $\xi_0=0$ and $\xi_{-0}=2-N$.

\medskip

The proof of Theorem \ref{thm} will be an adaptation of \cite[Theorem 6.4]{dauge_cones} (see also \cite[Chapters 5-6]{KozlovMazya}) and a localization argument. 

Specifically, we first study the asymptotics of compactly supported weak solutions to the Neumann problem in the infinite cone
\begin{equation}\label{Neu:tuttosigma}
    \begin{cases}
        -\Delta u=f\quad &\text{in $\Sigma_D$}
        \\
        \partial_\nu u=0\quad&\text{on $\partial \Sigma_D$}
    \end{cases}
\end{equation}
and of solution to the Dirichlet problem in the infinite cone
\begin{equation}\label{Dir:tuttosigma}
    \begin{cases}
        -\Delta u=f\quad &\text{in $\Sigma_D$}
        \\
         u=0\quad&\text{on $\partial \Sigma_D$.}
    \end{cases}
\end{equation}

This is the content of the following

\begin{theorem}\label{Dauge}
   Let $D\subsetneq \mathbb{S}^{N-1}$ be a domain of class $C^2$ such that  $\Sigma_D $ is a Lipschitz cone. Let $q\ge\frac{2N}{N+2}$ ($q>1$ if $N=2$)  be such that, upon setting $\eta=\eta_q=2-\frac{N}{q}$, $\eta(\eta+N-2)$ is not an eigenvalue of $L_D$.
      Let $u\in H^1(\Sigma_D)$ be a compactly supported solution to \eqref{Neu:tuttosigma} (resp \eqref{Dir:tuttosigma}) with $f\in L^q(\Sigma_D)$ and compactly supported. Then $u$ admits the finite expansion
    \begin{equation}\label{finite:exp}
    u=u_0+\sum_{\substack{i\geq 0\\\xi_i<\eta}}c_iP_i \quad\bigg(\text{resp. }  u=u_0+\sum_{{\substack{i\geq 1\\\xi_i<\eta}}}c_iP_i\bigg).
    \end{equation}
    where $c_i\in\R$, $\nabla^2u_0\in L^{q}(\Sigma_D)$, $|x|^{-1}\nabla u_0\in L^q(\Sigma_D)$ and $|x|^{-2}u_0\in L^q(\Sigma_D)$, $P_i$ is a harmonic function on $\Sigma_D$ of the form
    $$ P_i=r^{\xi_i}\mathfrak{a}_i(\theta),\text{ for some }\mathfrak{a}_i\in H^1(D), (\text{resp.} H_0^1(D)), $$
$P_i$ satisfies the same boundary conditions as $u$, and the sum in \eqref{finite:exp} may be empty.
\end{theorem}

\begin{proof}
 We introduce the new coordinates
   \begin{equation}\label{trasformazione}
    t:=\log|x|, \qquad \theta:=\frac{x}{|x|}\in D,\qquad x=\theta e^t
    \end{equation}
    and the following functions on $\R\times D$:
    \begin{equation}\label{def:vg}
    v(t,\theta)\coloneqq u(\theta e^t)=u(x), \qquad g(t,\theta) \coloneqq e^{2t}f(\theta e^t)=e^{2t}f(x).
    \end{equation}
It is known (see \cite[Sections 5.3, 6.1]{KozlovMazya}) that the change of variables \eqref{trasformazione} induces an isomorphism 
$$L^q(\Sigma_D)\to W_\zeta^{0,q}(\R\times D)\qquad\text{with }\zeta=-N/q
$$ 
where 
$$
W^{k,q}_\zeta (\R\times D) \coloneqq \{h:\,e^{-\zeta t}\,h\in W^{k,q}(\R\times D)\},\quad  k\in \N\cup \{0\}.
$$
We also set
$$
\mathcal{H}^1_\zeta:=\begin{cases}
    \{w\in W^{1,2}_\zeta(\R\times D)\}\quad&\text{in the case of Neumann problems}
    \\
    \{w\in W^{1,2}_\zeta(\R\times D):\,w=0 \text{ on }\R\times\partial D\} \quad&\text{in the case of Dirichlet problems,}
\end{cases}
$$
and 
$$
\mathcal{H}^{-1}_\zeta:=W^{-1,2}_\zeta(\R\times D)=(\mathcal{H}^1_{-\zeta})'.
$$
The functions $g,v$ satisfy the following properties:
\vspace{0.2cm}

\begin{property*}[1] (see \cite[Lemma 6.6]{dauge_cones})
  Let $\eta=\eta_q=2-\frac{N}{q}$, and $v,g$ be given by \eqref{def:vg}. Then
    $$
    \forall \zeta\le\eta, \qquad g\in W_\zeta^{0,q}(\R\times D)
    $$
  \begin{equation}
  \label{6.2}
    \forall \zeta<\eta, \qquad g\in \mathcal{H}_\zeta^{-1} 
  \end{equation}
    and 
   \begin{equation}\label{6.3}
    \forall \zeta<1-N/2, \qquad v\in \mathcal{H}_\zeta^{1}.
  \end{equation}
    Moreover $v$ is solution to
    \begin{equation}\label{6.4}
   - (\partial_t^2+(N-2)\partial_t+\Delta_{\mathbb S^{N-1}})v=g.
  \end{equation}
  \end{property*}
\vspace{0.2cm}
  
\begin{property*}[2](see \cite[Lemma 6.8]{dauge_cones})
        If $\zeta\in\R$ is such that
        $$
        \zeta(\zeta+N-2) \text{ is not an eigenvalue of }L_D
        $$
         then the operator $-(\partial_t^2+(N-2)\partial_t+\Delta_{\mathbb S^{N-1}})$ induces an isomorphism
        $$
        A^\zeta:\mathcal{H}^1_\zeta\to\mathcal{H}^{-1}_\zeta.
        $$
\end{property*}

\begin{property*}[3] (see \cite[Lemma 6.9]{dauge_cones})
    Let $\zeta_1,\zeta_2\in\R$, $\zeta_1<\zeta_2$ be such that
    $$
    \zeta_j\,(\zeta_j+N-2)\text{ is not an eigenvalue of }L_D\text{ for }j=1,2.
    $$
    For a given $h\in\mathcal{H}^{-1}_{\zeta_1}\cap\mathcal{H}^{-1}_{\zeta_2}$, let
    \begin{equation*}
        w_j\coloneqq (A^{\zeta_j})^{-1}h\quad\text{and}\quad u_j(x)=u_j(\theta\,e^t)\coloneqq  w_j(t,\theta).
    \end{equation*}
for $j=1,2$.   Then we have
   \begin{equation}\label{sum3}
    u_1-u_2=\sum_{\xi_i\in(\zeta_1,\zeta_2)} c_iP_i    
    \end{equation}
    where $c_i\in\R$, and $P_i$ fulfills
    $$
    \begin{cases}
        \Delta P_i=0\quad&\text{in $\Sigma_D$}\\
       
        \partial_\nu P_i=0 \quad(\text{resp. }  P_i=0)\quad&\text{on $\partial\Sigma_D$}\\
    \end{cases}
    $$
    and have the explicit form, for $i\in\mathbb Z$, $i\ne 0$:
    $$
P_i=r^{\xi_i}\,\mathfrak{a}_i(\theta)\text{ with }\mathfrak{a}_i\in H^1(D), (\text{resp. }H_0^1(D)). $$
In the Neumann case the sum \eqref{sum3} may contain $P_0$ and $P_{-0}$, where $P_0=1$ and $P_{-0}=r^{2-N}$ for  $N\geq 3$ and $P_{-0}=\log(r)$ for $N=2$.

In particular, if for every $ \xi\in(\zeta_1,\zeta_2),\, \xi(\xi+N-2)$ is not an eigenvalue of $L_D$, then  $w_1\equiv w_2$, hence $u_1\equiv u_2$.
\end{property*}
    % Property (3) is the analogue of  \cite[Lemma 6.9]{dauge_cones}, which in turn is derived from \cite{Kondratev}-- see also \cite{KozlovMazya}.
% For the sake of completeness, we provide further details of the proof.
\vspace{0.2cm}

\noindent \textbf{Conclusion.}
Let $g$ and $v$ be given by \eqref{def:vg}. Since by hypothesis $\eta(\eta+N-2)$ is not an eigenvalue of $L_D$, by the discreteness of the spectrum we may find $\zeta_2<\eta$ such that $\zeta(\zeta+N-2)$ is not an eigenvalue of $L_D$ for all $\zeta\in[\zeta_2,\eta]$. According to \eqref{6.2} and Properties (2)-(3), for $\zeta\in[\zeta_2,\eta)$ we may thus define the functions
\begin{equation*}
    w(t,\theta)\coloneqq (A^{\zeta})^{-1}g,\quad u_0(x)=u_0(\theta\,e^t)\coloneqq w(t,\theta).
\end{equation*}
% choose there exists $\zeta_0<\eta$ such that
% $$
% \forall\zeta\in[\zeta_0,\eta], \qquad \zeta(\zeta+N-2)\text{ is not an eigenvalue of }-\Delta_D.
% $$
% Combining \eqref{6.2} with Property (3), we conclude that, for every $\zeta\in[\zeta_0,\eta]$, the function $(A^\zeta)^{-1}g$ is independent of $\zeta$. We denote this common function by $w$. We then set
% $$u_0(x)=u_0(\theta e^t)\coloneqq w(t,\theta).$$
Let $\zeta_1<\min\{1-N/2,\zeta_2\}$ be such that for all $\zeta\in[\zeta_1,1-N/2)$, $\zeta(\zeta+N-2)$ is not an eigenvalue of $L_D$. Then by Property (2), \eqref{6.3} and \eqref{6.4} we have $(A^{\zeta_1})^{-1}g=v$.  We now employ Property (3) with  $\zeta_1,\zeta_2$ as above and, recalling \eqref{def:vg}, we infer

% We then fix $\zeta_1<\zeta_2$ such that $\zeta_1<1-N/2$, and $\zeta_1<\eta$, and make use of Property (3) as to obtain
$$
u=u_0+\sum_{\xi_i\in(\zeta_1,\zeta_2)}c_iP_i,
$$
Now, all $\xi_{-j}$, $j\geq 1$ lie below $1-N/2$ and hence, from our choice of $\zeta_1$, they lie below $\zeta_1$. In the Neumann case and $N\geq 3$ this includes also $\xi_{-0}=2-N$, hence only $\xi_i$, $i\geq 0$ can occur. Therefore the sum is over $i\geq 0$ in the Neumann case and $i\geq 1$ in the Dirichlet case, both subject to $\xi_i<\eta$. The sum may be empty. In the Neumann case and $N=2$, we have that $P_{-0}=\log(r)$ may enter the sum.  However, this cannot be the case since $u\in H^1(\Sigma_D)$.

Therefore, the proof will be concluded once we show that $\nabla^2 u_0\in L^q({\Sigma}_D)$, $ |x|^{-1}\nabla u_0\in L^q(\Sigma_D) $ and $ |x|^{-2}u_0\in L^q(\Sigma_D)$; by the change of variables \eqref{trasformazione}, it is immediate to see that this is equivalent to prove that
\begin{equation}\label{dauge_4}
e^{-\eta t}w\in W^{2,q}(\R\times D).
\end{equation}
Since $D$ is of class $C^2$, standard elliptic regularity holds in the cylinder $\mathbb R\times D$, hence \cite[Hypothesis 6.2]{dauge_cones} is satisfied. At this point,  \eqref{dauge_4} follows from the fourth step in the proof of \cite[Theorem 6.4]{dauge_cones}. 
\end{proof}

\begin{proof}[Proof of Theorem \ref{thm}]

Let $u$ be a weak solution to \eqref{eq:lapl:neu} (resp \eqref{eq:lapl:dir}). We first perform a localization argument to reduce to the setting of Theorem \ref{Dauge}. To this end, let $\phi$ be a radially symmetric function such that $\phi\in C^\infty_c(B_{3/2})$ and $\phi\equiv 1$ in $B_{5/4}$. Define $\tilde{u}=u\,\phi$, so that $\tilde{u}\in H^1(\Sigma_D)$ with support bounded in the radial direction  is solution to
\begin{equation}\label{solvesu}
    -\Delta \tilde{u}=f\,\phi-2\,\nabla u\cdot \nabla \phi-u\,\Delta\phi.
\end{equation}
Moreover, since $\partial \Sigma_D\setminus\{0\}$ is of class $C^2$, and $f\in L^q(\Sigma_D\cap B_2)$, by standard elliptic regularity theory \cite{GT, L13}, we have that $u\in W^{2,q}\big(\Sigma_D\cap (B_{3/2}\setminus B_{1})\big)$. As $\mathrm{supp}(\nabla\phi)\cup\mathrm{supp}(\nabla^2\phi)\Subset B_{3/2}\setminus B_1$, this in turn implies that the right-hand side of \eqref{solvesu} is of class $L^q(\Sigma_D)$ and is compactly supported.

Finally, if $u$ solves the Dirichlet problem \eqref{eq:lapl:dir}, then clearly $\tilde{u}=0$ on $\partial \Sigma_D$. On the other hand,  the radial symmetry of $\phi$ implies that $\partial_\nu \phi=0$ on $\partial \Sigma_D$; hence, if $u$ solves \eqref{eq:lapl:neu}, we have that $\partial_\nu \tilde{u}=(\partial_\nu u)\,\phi+u\,\partial_\nu \phi=0$ on $\partial\Sigma_D$. We conclude that $\tilde{u}$ satisfies the assumptions of Theorem \ref{Dauge}, with $\tilde{u}\equiv u$ on $\Sigma_D\cap B_1$. For this reason, and to simplify the notation, we continue to denote by $u$ the function $\tilde{u}$.
\vspace{0.1cm}

From Theorem \ref{Dauge} with Neumann (resp. Dirichlet) boundary conditions we obtain that
\begin{equation}\label{sum}
u=u_0+\sum_{\substack {i\geq 0\\\xi_i<\eta}} c_iP_i\quad\bigg(\text{resp. }  u=u_0+\sum_{\substack {i\geq 1\\\xi_i<\eta}}c_iP_i\bigg),
\end{equation}
where $c_i\in\R$, $\nabla^2u_0\in L^q(\Sigma_D)$, $|x|^{-1}\nabla u_0\in L^q(\Sigma_D)$ and $|x|^{-2}u_0\in L^q(\Sigma_D)$. The weighted estimates imply that $u_0\in W^{2,q}(\Sigma_D\cap B_1)$. 
Moreover $P_i$ are harmonic functions in $\Sigma_D$ with Neumann (resp. Dirichlet)  boundary conditions, of the form $P_i=r^{\xi_i}\,\mathfrak{a}_i(\theta)$, where we recall that, for $i\geq 0$ in the Neumann case and $i\geq 1$ in the Dirichlet case,
%By writing the Laplacian in polar coordinates, they necessarily satisfy
%$$
%\partial^2_{rr}P_i+(N-1)\frac{\partial_r P_i}{r}+\frac{\Delta_{\mathbb S^{N-1}} P_i}{r^2}=0
%$$
%which implies that $\mathfrak{a}_i$ solve
%$$
%-\Delta_D \mathfrak{a}_i(\theta)=\xi_i(\xi_i+N-2)\,\mathfrak{a}_i(\theta).
%$$
%This means that $\mathfrak{a}_i(\theta)$  is a Neumann (resp. Dirichlet) eigenfunction of $-\Delta_{\mathbb S^{N-1}}$, corresponding to the eigenvalue $\lambda_i=\xi_i(\xi_i+N-2)$, and thus
$$
\xi_i=\frac{-(N-2)+\sqrt{4\lambda_i(D)+(N-2)^2}}{2}.
$$
To complete the proof, it remains to analyze the regularity of the sum in \eqref{sum}. In view of the regularity of $u_0\in W^{2,q}(\Sigma_D\cap B_1)$, this is equivalent to studying the regularity of the functions $P_i=r^{\xi_i}\mathfrak{a}_i(\theta)$. Recall that $\mathfrak{a}_i$ belongs to the eigenspace associated with $\lambda_i(D)$, hence it is enough to study the regularity of $r^{\xi_i}Y_i(\theta)$.

First, note that if $\lambda_1(D)>N\left(2-\frac{N}{q}\right)\left(1-\frac{1}{q}\right)$, then $\xi_1>2-\frac{N}{q}=\eta$, hence  \eqref{sum} reduces to $u_0$ plus at most a constant in the Neumann case, and to $u_0$ in the Dirichlet case, and this proves Assertion (1).

%Since $Y_i$ are $C^2$ functions, differentiation with respect to the angular variables does not introduce any singular behavior. Also, derivatives of order \(k\in \N\), $k\leq 2$, satisfy
%\[
%|\nabla^k P_i(x)|\sim C r^{\xi_i-k},
%\]
%hence via a simple change of polar coordinates we deduce that $
%\nabla^2 P_i\in L^q(\Sigma_D\cap B_1) $ if and %only if $q\,(\xi_i-2)+N>0 $ (or equivalently,  $
%\xi_i>2-\frac{N}{q}$), then
%. This means
%$$
%\xi_1=\frac{-(N-2)+\sqrt{4\lambda_1+(N-2)^2}}%{2}>2-\frac{N}{q},
%$$
%but the sum \eqref{sum} ranges over $\xi_<\eta=2-N/q$, thus $u=u_0$ and we get Assertion (1). 
Now we prove points $(2),(3)$ and $(4)$.
Since $q>N$, from Stein extension Theorem applied to $u_0$ in $\Sigma_D\cap B_1$, and Sobolev embedding Theorem, we deduce $u_0\in C^{1,\beta}(\overline{\Sigma}_D\cap B_1)$, $\beta=1-\frac{N}{q}$. This property coupled with the fact that  $|x|^{-1}\nabla u_0,|x|^{-2}u_0\in L^q(\Sigma_D)$ yields that $u_0(0)=0$ and $\nabla u_0(0)=0$.

Assume now that $\lambda_1(D)=N-1$. Thus $0=\xi_0$,  $1=\xi_1\leq \xi_i$ for all $i\geq 1$,  and this implies  the Lipschitz continuity of $P_i=r^{\xi_i}\,Y_i(\theta)$  (but in general not the $C^1$ regularity due to the possible presence of terms of the form $rY_1(\theta)$, see Remark \ref{finalrmk}). Since $u_0\in C^{1,\beta}(\overline{\Sigma}_D\cap B_1)$, this implies
$u\in C^{0,1}(\overline{\Sigma}_D\cap B_1)$.

Finally, in the case $\lambda_1>N-1$ we have $\xi_1>1$; hence $\xi_0=0$ and $\xi_i>1$ for $i\geq 1$, so that each term in the summation of \eqref{sum}
is of class $C^{1,\xi_i-1}(\overline{\Sigma}_D\cap B_1)$ if also $\xi_i<2$, while it is of class at least $C^{1,\gamma}(\overline{\Sigma}_D\cap B_1)$ for all $\gamma<1$ if $\xi_i\geq 2$ (this is the regularity of $Y_i$ when $D$ is of class $C^2$). Therefore
$u\in C^{1,\alpha}(\overline{\Sigma}_D\cap B_1)$ with
$\alpha=\min\{1-N/q,\xi_1-1\}\in(0,1)$. Moreover, $\nabla P_i(0)=0$, which together with the fact that $\nabla u_0(0)=0$ yields
$\nabla u(0)=0$. This concludes the proof.
\end{proof}

 %\todo[inline]{Nel seguente Remark vengono usate le lettere $v,p$ per indicare oggetti diversi da quelli indicati fino ad ora. Io non penso che ci sia ambiguità, ma valutiamo.}

% \todo[inline]{CA: forse spostare questo remark nell'intro}
\begin{remark}\label{finalrmk}
    \rm{When $\lambda_1=N-1$, Theorem \ref{thm} states that $u$ is  Lipschitz continuous at the origin. This follows from the regularity of the terms $rY(\theta)$, where $Y(\theta)$ is an eigenfunction corresponding to the eigenvalue $\lambda_1=N-1$,  which appears in the decomposition \eqref{decomposition0} of $u$, since all the other ones have derivatives that extend continuously at the vertex (they actually vanish at the vertex). One may ask then in which cases the gradient of $u$ can be extended continuously at the vertex to get more regularity. This happens in extremely specific situations. In fact, $\nabla (rY(\theta))=Y(\theta)\theta+\nabla_DY(\theta)$. Hence it extends continuously at the vertex if and only if  $Y(\theta)\theta+\nabla_DY(\theta)=v$ for some fixed $v\in\mathbb R^N$. This implies, taking the scalar product with $\theta$, that $Y(\theta)=\langle v,\theta\rangle$, that is $rY(\theta)=\langle x,v\rangle$ is a linear function (recall $x=r\theta$ in spherical coordinates), which we know to be an eigenfunction on $\mathbb S^{N-1}$ with eigenvalue $N-1$. Moreover, $Y(\theta)$ must satisfy Dirichlet or Neumann conditions on $\partial D$. In other words, this happens if and only if the restriction of a linear function on $D$ satisfies Dirichlet or Neumann boundary conditions on $\partial D$. For the Dirichlet case, this happens if and only if $D$ is a hemisphere. In all the other cases $C^1$-regularity cannot be guaranteed for arbitrary solutions. For the Neumann case we have more domains. We have that on $\partial D$, $\partial_\nu \langle \theta,v\rangle=\langle\nu,v\rangle$, which means that the restriction of $v$ to $\mathbb S^2$ must be tangential to $\partial D$. Let $\pi_v$ be the hyperplane in $\mathbb R^N$ through $0$ orthogonal to $v$ and let $E_v=\pi_v\cap\mathbb S^{N-1}$ be the corresponding equator in $\mathbb S^{N-1}$. Let $N,S$ the opposite poles lying on the line through $0$ parallel to $v$. Let $p\in E_v$ and set $\gamma(p)$ the geodesic segment through $p$ connecting $N$ and $S$. If $\omega$ is a smooth open set of $E_v$, then $D=\cup_{p\in \omega}\gamma(p)$ is a domain of $\mathbb S^{N-1}$ (it is a ``double cone'' with vertices at the poles). Assuming it is Lipschitz, then the restriction of $\langle x,v\rangle$
 is a Neumann eigenfunction with eigenvalue $N-1$. These are all such domains. In the case of $\mathbb S^2\subset\mathbb R^3$ these are all the domains bounded by two meridians (i.e., they are hemispheres or lunes). If we require the domains to be smooth, the only possibility is that $D$ is a hemisphere.}
\end{remark}

%\todo[inline]{Aggiungete anche i vostri dettagli istituzionali}

\bigskip{}{}

 \par\noindent {\bf Acknowledgments.}  C.A.  Antonini is supported by the Italian Ministry of University and
Research (MUR) through the FIS 2 project "SiGmA - Singularities in Geometric Analysis: Minimal Surfaces
and Mean Curvature Flows", project code FIS-2023-02962, CUP G53C25000120001.

\bigskip{}{}

 \par\noindent {\bf Data availability statement.} Data sharing not applicable to this article as no datasets were generated or analyzed during the current study. 

\section*{Compliance with Ethical Standards}\label{conflicts}

\par\noindent
{\bf Funding}. Research partially funded by GNAMPA   of the Italian INdAM - National Institute of High Mathematics (grant number not available).
L.P. has been supported  by the project ``Analisi Geometrica e Teoria  Spettrale su variet\'a Riemanniane ed Hermitiane'' of the INdAM GNSAGA.

\bigskip
\par\noindent
{\bf Conflict of Interest}. The authors declare that there is no conflict of interest.

\end{document}